\documentclass[a4paper,10pt]{amsart}
\usepackage{txfonts,amsfonts}
\usepackage[arrow,matrix]{xy}
\usepackage{amsmath,amssymb,amscd,bbm,amsthm,mathrsfs,dsfont,bm}
\usepackage{extarrows}
\usepackage{latexsym}
\usepackage{graphicx}
\usepackage{color}
\usepackage{xy}
\usepackage{leftidx}
\usepackage{pifont}
\input xy
\input xypic
\xyoption{all}

\allowdisplaybreaks
\newtheorem{theorem}{Theorem}[section]
\newtheorem{lemma}[theorem]{Lemma}
\newtheorem{proposition}[theorem]{Proposition}
\newtheorem{corollary}[theorem]{Corollary}

\theoremstyle{definition}
\newtheorem{definition}[theorem]{Definition}

\newtheorem{remark}[theorem]{Remark}
\newtheorem{solution}[theorem]{Solution}
\numberwithin{equation}{section}
 \DeclareMathOperator{\Ker}{Ker}
\DeclareMathOperator{\Image}{Im}

\DeclareMathOperator{\uExt}{\underline{Ext}}

\DeclareMathOperator{\uHom}{\underline{Hom}}

\DeclareMathOperator{\id}{id}

\DeclareMathOperator{\hdet}{hdet}
\def\bt{\begin{theorem}}
\def\et{\end{theorem}}
\def\bl{\begin{lemma}}
\def\el{\end{lemma}}
\def\br{\begin{remark}}
\def\er{\end{remark}}
\def\bc{\begin{corollary}}
\def\ec{\end{corollary}}

\begin{document}
%\title[Nakayama automorphisms of graded Ore extensions of Koszul AS-regular algebras]{Nakayama automorphisms of graded Ore extensions of Koszul Artin-Schelter regular algebras with non-trivial derivations
%}

\title{Nakayama automorphisms of graded Ore extensions of Koszul Artin-Schelter regular algebras
}

\author{Y. Shen}
\address{Shen: Department of Mathematics, Zhejiang Sci-Tech University, Hangzhou 310018, China}
\email{yuanshen@zstu.edu.cn}

\author{Y. Guo}
\address{Guo: Department of Mathematics, Zhejiang Sci-Tech University, Hangzhou 310018, China}
\email{2275045039@qq.com}

\date{}

\begin{abstract}
Let $A$ be a Koszul Artin-Schelter regular algebra, $\sigma$ a graded automorphism of $A$ and $\delta$ a degree-one $\sigma$-derivation of $A$. We introduce an invariant for $\delta$ called the $\sigma$-divergence of $\delta$. We describe the Nakayama automorphism of the graded Ore extension $B=A[z;\sigma,\delta]$ explicitly using the $\sigma$-divergence of $\delta$, and construct a twisted superpotential $\hat{\omega}$ for $B$ so that it is a derivation quotient algebra defined by $\hat{\omega}$. We also determine all graded Ore extensions of noetherian Artin-Schelter regular algebras of dimension 2 and compute their Nakayama automorphisms.
\end{abstract}

\subjclass[2010]{16S36,16S37,16E65,16S38,16W50}

\keywords{Koszul Artin-Schelter regular algebras, Graded Ore extensions, Nakayama automorphisms, Twisted superpotentials, Calabi-Yau}

\maketitle
%\tableofcontents

\section*{Introduction}
To understand Artin-Schelter regular algebras has been a main topic in the study of noncommutative algebras and noncommutative projective geometry since the late 1980s. These algebras are exactly  connected graded skew Calabi-Yau algebras (\cite{RRZ1}), and possess a kind of automorphisms called Nakayama automorphisms. Such automorphisms are an important tool to study  Hopf actions, noncommutative invariant theory and so on (\cite{CWZ,LMZ,RRZ2,SL}). However, the computation of Nakayama automorphisms is always hard. There is a plenty of work to provide different methods to solve this problem (see \cite{HVZ,LM,LWW,LMZ,LMZ1,RRZ1,RRZ2,SL,SZL,V,ZSL,ZVZ}).

Obtaining new Artin-Schelter regular algebras from known ones is a common approach, and the methods include Ore extensions, double Ore extensions and regular normal extensions. How the Nakayama automorphisms behave under those extensions is an interesting and concerned problem. The behavior of Nakayama automorphisms under regular normal extensions was studied in \cite{RRZ1,ZSL}. The Nakayama automorphisms of trimmed double Ore extensions of Koszul Artin-Schelter regular algebras were described in \cite{ZVZ}. Earlier, Liu, Wang and Wu considered the case of Ore extensions in a general setting (\cite{LWW}). They proved that if $A$ is (not necessarily connected graded) skew Calabi-Yau with a Nakayama automorphism $\mu_A$, then the Ore extension $B=A[z;{\sigma},\delta]$ has a Nakayama automorphism $\mu_B$ satisfying ${\mu_B}_{\mid A}={\sigma}^{-1}\mu_A$ and $\mu_B(z)=\lambda z+b$ for some $\lambda,b\in A$. It is natural to ask what $\lambda$ and $b$ are. Restricted on graded Ore extensions, Zhu, Van Oystaeyen and Zhang showed  $\lambda=\hdet({\sigma})$ if $A$ is Koszul Artin-Schelter regular in \cite{ZVZ}, and Zhou, Lu and the first named author showed the same equality if $A$ is noetherian Artin-Schelter regular generated in degree 1 in \cite{SZL}, where $\hdet$ is the homological determinant introduced by  J{\o}rgensen and Zhang (see \cite{JZ}).

In fact, $\lambda$ can be determined by trimmed Ore extensions, namely, the ${\sigma}$-derivation ${\delta}$ is trivial, by the filtered-graded technique. However, $b$ is related to ${\delta}$, and seems mysterious for us. Recently, Liu and Ma described  $\lambda$ and $b$ explicitly for all (ungraded) Ore extensions if $A$ is a polynomial algebra in \cite{LM}. It inspires us to  make a progress in the description of the parameter $b$. The main goal of this paper is to describe the Nakayama automorphisms of graded Ore extensions of Koszul Artin-Schelter regular algebras specifically.

Let $A=T(V)/(R)$ be a Koszul Artin-Schelter regular algebra of dimension $d$, ${\sigma}$ a graded automorphism of $A$ and ${\delta}$ a degree-one ${\sigma}$-derivation of $A$. In order to realize our desire, it requires a way to handle the ${\sigma}$-derivation $\delta$.  Our approach is to construct a pair $(\{\delta_{i,r}\},\{\delta_{i,l}\})$ of two sequences of linear maps for ${\delta}$:
$$
\delta_{i,r}:W_i\to W_i\otimes V,\qquad \delta_{i,l}: W_i\to V\otimes W_i,
$$
where $W_1=V$ and $W_j=\cap_{s=1}^{j} V^{\otimes s}\otimes R\otimes V^{j-s-2}$ for any $i\geq 1,j\geq 2$, by the Koszul complex of the graded trivial module $k_A$ (see Lemma \ref{lemma: construction of delta_r} and Lemma \ref{lemma: construction of delta_l}). In general, sequence pair $(\{\delta_{i,r}\},\{\delta_{i,l}\})$  is not unique for ${\delta}$. However, there is a unique pair $(\delta_r,\delta_l)$ of elements in $V$ coming from the action of $\delta_{d,r}$ and $\delta_{d,l}$ on $W_d$ for each sequence pair $(\{\delta_{i,r}\},\{\delta_{i,l}\})$,  since $W_d$ is $1$-dimensional. We find that $\delta_{r}+\mu_A{\sigma}^{-1}(\delta_l)$ is independent on the choices of sequence pairs for $\delta$, where $\mu_A$ is the Nakayama automorphism of $A$ (see Corollary \ref{coro: relation between delta_r and delta l}). We call this invariant the \emph{${\sigma}$-divergence} of ${\delta}$, denoted by $\nabla_{{\sigma}}\cdot{\delta}$. If $A$ is a polynomial algebra and $\sigma$ is the identity map $\id_A$, then $\nabla_{{\id_A}}\cdot{\delta}$ is precisely the usual divergence of $\delta$.

It is well known that the graded Ore extension $B=A[z;{\sigma},{\delta}]$ is also a Koszul Artin-Schelter regular algebra of dimension $d+1$. Using a sequence pair $(\{\delta_{i,r}\},\{\delta_{i,l}\})$ for ${\delta}$, we compute the Yoneda product of the Ext-algebra $E(B)$ of $B$. Since $E(B)$ is a graded Frobenius algebra and its Nakayama automorphism is dual to the one of $B$ (see \cite{RRZ2,V}), then we obtain the main result of this paper.

\begin{theorem}\emph{(Theorem \ref{thm: nakayama automorphism of ore extension})}
Let $B=A[z;{\sigma},{\delta}]$ be a graded Ore extension of a Koszul Artin-Schelter regular algebra $A$, where ${\sigma}$ is a graded automorphism of $A$ and $\delta$ is a degree-one $\sigma$-derivation. Then the Nakayama automorphism $\mu_B$ of $B$ satisfies
\begin{align*}
{\mu_{B}}_{\mid A}={\sigma}^{~-1}\mu_A,\qquad
\mu_B(z)=\hdet({\sigma})\, z+\nabla_{{\sigma}}\cdot{\delta},
\end{align*}
where $\mu_A$ is the Nakayama automorphism of $A$.
\end{theorem}

It is also well-known that a Koszul Artin-Schelter regular algebra $A$ is determined by a twisted superpotential $\omega$, that is, $A$ is isomorphic to a derivation quotient algebra defined by $\omega$ (see \cite{BSW,DV}). He, Van Oystaeyen and Zhang constructed a new twisted superpotential for a graded Ore extension of $A$ in two special cases (\cite{HVZ,HVZ1}). In fact, any sequence pair $(\{\delta_{i,r}\},\{\delta_{i,l}\})$ for ${\delta}$ provides us assistance to construct  a new twisted superpotential $\hat{\omega}$ for any graded Ore extension $B=A[z;\sigma,\delta]$ from $\omega$, and it is independent on the choices of sequence pairs. Hence, $B$ is isomorphic to a derivation quotient algebra defined by $\hat{\omega}$ (see Theorem \ref{thm: twisted superpotential for B}).

As an application of the main result, we determine all graded Ore extensions of noetherian Koszul Artin-Schelter regular algebras of dimension 2 and compute their Nakayama automorphisms. There is an interesting observation about Calabi-Yau property.

\begin{theorem}\emph{(Theorem \ref{thm: CY for 2-dim})} Assume the base field $k$ is of characteristic $0$.
Let $A=k\langle x_1,x_2\rangle/(r)$ be a noetherian Artin-Schelter regular algebra of dimension 2, and $B$ a graded Ore extension $A[z;{\sigma},{\delta}]$. %Write $\mu_A$ for the Nakayama automorphism of $A$.
\begin{enumerate}
\item  Suppose $A$ is commutative, then $B$ is Calabi-Yau if and only if ${\sigma}=\id_A$, and $$   {\delta}(x_1)=l_1x_1^2-2l_4x_2x_1+l_2x_2^2,\qquad
    {\delta}(x_2)=l_3x_1^2-2l_1x_2x_1+l_4x_2^2,
    $$ for some $l_1,l_2,l_3,l_4\in k$.
\item  Suppose $A$ is noncommutative, then $B$ is Calabi-Yau if and only if ${\sigma}$ is the Nakayama automorphism of $A$.
\end{enumerate}
\end{theorem}

The graded automorphism group for the commutative noetherian Artin-Schelter regular algebra of dimension 2 is much bigger than a noncommutative one, as well as the set of $\sigma$-derivations. So there are more conditions for graded Ore extensions of the commutative one being Calabi-Yau.  But it is still surprised that there is no any restriction on $\delta$ for $A[z;\sigma,\delta]$ being a Calabi-Yau algebra in case $A$ is a noncommutative noetherian Artin-Schelter regular algebra of dimension 2. It is natural to ask whether this result holds without the noetherian assumption, or even for all noncommutative Koszul Artin-Schelter regular algebras.

This paper is organized as follows. In Section 1, we recall some definitions and properties, especially for Koszul Artin-Schelter regular algebras. In Section 2, we construct a sequence pair for any $\sigma$-derivation of a Koszul algebra， where $\sigma$ is a graded automorphism, and discuss the relations between different sequence pairs. In Section 3, we focus on the Koszul Artin-Schelter regular algebras. We prove there is an invariant for each $\sigma$-derivation, compute the Yoneda product of the Ext-algebra of any graded Ore extension, obtain the main result about Nakayama automorphisms and construct a twisted superpotential for a graded Ore extension. In Section 4, we apply our main result to commutative polynomial algebras and Koszul Artin-Schelter regular algebras of dimension 2.

Throughout the paper, $k$ is a fixed field. All vector spaces and algebras are over $k$. Unless otherwise stated, the tensor product $\otimes$ means $\otimes_k$.
\section{Preliminaries}\label{Section preliminaries}

A graded algebra $A=\oplus_{i\in\mathbb{Z}} A_i$ is called \emph{locally finite} if  $\dim A_i<\infty$ for any $i\in\mathbb{Z}$. A locally finite graded algebra $A$ is called \emph{connected} if $A_i=0$ for any $i<0$ and $\dim A_0=1$. In this case, write $\varepsilon_A$ for the augmentation from $A$ to $k$. Let $M=\oplus_{i\in \mathbb{Z}}M_i$ be a right graded $A$-module. Then $n$-th \emph{shift} of $M$ is a right graded $A$-module $M(n)$ with the homogenous space $M(n)_i=M_{n+i}$ for all $i\in\mathbb{Z}$. Let $N$ be a graded $A$-bimodule and $\mu$  a graded automorphism of $A$. The graded twisted  $A$-bimodule $N^{\mu}$ is a graded $A$-bimodule with the $A$-action $a\cdot x\cdot b=ax\mu(b)$ for any $a,b\in A,x\in N$.

For a connected graded algebra $A$, there exists a minimal graded free resolution of the right graded trivial module $k_A$ of $A$,
\begin{equation}\label{general Resolution of k_A}
\cdots\xlongrightarrow{} P_{d}\xlongrightarrow{\partial_d} \cdots\xlongrightarrow{\partial_3} P_{2}\xlongrightarrow{\partial_2} P_1\xlongrightarrow{\partial_1} P_0\xlongrightarrow{\partial_0} k_A\xlongrightarrow{} 0,
\end{equation}
namely, $\Ker \partial_i\subseteq P_{i}A_{\geq 1}$ for any $i\geq 0$. Then the graded vector space $E(A)=\bigoplus_{i\geq0}\uExt^i_A(k_A,k_A)=\bigoplus_{i\geq0}\uHom_A(P_i,k)$ equipped with the Yoneda product is a connected graded algebra, called the \emph{Ext-algebra} of $A$. In the sequel, We denote the Yoneda product by ``$\ast$''.

%\begin{lemma}\label{lemma: maps between ext algebras}\cite[Theorem 1]{SWZ}
%Let $A,A'$ be two connected graded algebras, and $f:A\to A'$ a graded algebra homomorphism. Then there exists a graded algebra homomorphism $E(f):E(A')\to E(A)$.
%\end{lemma}

Let $V$ be a finite dimensional vector space. Write $\tau:V\otimes V\to V\otimes V$ for the usual twisting map. We adopt the notation in \cite{HVZ} of a sequence of linear endmorphisms of $V^{\otimes d}$ for any $d\geq 2$:
$$
\tau_{d}^0=\id_V^{\otimes d},\quad %\tau_{d}^1=\tau\otimes\id_V^{\otimes d-1},\cdots
\tau_d^{i}=(\id_V^{\otimes i-1}\otimes \tau\otimes\id_V^{\otimes d-i-1 })\tau_{d}^{i-1},\quad\text{for any }~ 1\leq i\leq d-1.
$$

\begin{definition}
Let $V$ be a finite dimensional vector space and $\nu:V\to V$ an isomorphism of vector spaces. If an element $\omega\in V^{\otimes d}$ for some $d\geq 2$ such that
$$
\omega=(-1)^{d-1}\tau_{d}^{d-1}(\nu\otimes\id_V^{\otimes d-1})(\omega),
$$
then $\omega$ is called a \emph{$\nu$-twisted superpotential}. In particular, $\omega$ is called a \emph{superpotential} if $\nu=\id_V$.
\end{definition}

Let $\nu$ be a linear automorphism of $V$ and integer $d\geq2$. Define the \emph{partial derivation} of a $\nu$-twisted superpotential $\omega\in V^{\otimes d}$ with respect to $\psi\in (V^*)^{\otimes i}$, where $(-)^*$ is the  $k$-dual of a space and $1\leq i\leq d$,  to be
$$
\partial_{\psi}(\omega)=(\id_V^{\otimes d-i}\otimes \psi)(\omega).
$$

\begin{definition}
Let $V$ be a finite dimensional vector space,  $\omega\in V^{\otimes d}$ a $\nu$-twisted superpotential for some linear automorphism $\nu$ of $V$ and $d\geq 2$. The \emph{$i$-th derivation quotient algebra} $\mathcal{A}(\omega,i)$ of $\omega$ is
$$
\mathcal{A}(\omega,i)=T(V)/(\partial_{\psi}(\omega),\psi\in (V^*)^{\otimes i}),
$$
where $T(V)$ is the tensor algebra over $V$.
\end{definition}

%Let $V$ be a finite dimensional vector space,
Let $T(V)$  be the tensor algebra over $V$ with the usual grading and $R$ a subspace of $V^{\otimes 2}$. Then $A=T(V)/(R)$ is a connected graded algebra, called a \emph{quadratic algebra}. Write $\pi_A$ for the canonical projection from $T(V)$ to $A$.

\begin{definition}
A quadratic algebra $A$ is called \emph{Koszul}, if the right graded trivial module $k_A$ admits a minimal graded free resolution (\ref{general Resolution of k_A}) such that the right graded $A$-module $P_i$ is generated by degree $i$ for any $i\geq0$.
\end{definition}

Let $A=T(V)/(R)$ be a Koszul algebra. Write $W_0=k$, $W_1=V$, $W_2=R$ and for $i\geq 3$,
$$
W_{i}=\bigcap_{0\leq s\leq i-2} V^{\otimes s}\otimes R\otimes V^{\otimes i-s-2}.
$$

Then the following Koszul complex is  a minimal graded free resolution of $k_A$,
$$
\cdots\xlongrightarrow{\partial_{d+1}^A} W_d\otimes A\xlongrightarrow{\partial^A_d}W_{d-1}\otimes A\xlongrightarrow{\partial^A_{d-1}}\cdots\xlongrightarrow{\partial^A_{3}}W_2\otimes A\xlongrightarrow{\partial^A_{2}}W_1\otimes A\xlongrightarrow{\partial^A_{1}}A\xlongrightarrow{\varepsilon_A}k_A\to0,
$$
where $\partial^A_{i}=(\id_V^{\otimes i-1}\otimes m_A)(\id_V^{\otimes i-1}\otimes{\pi_A}_{\mid V}\otimes\id_A)=\id_V^{\otimes i-1}\otimes m_A$ and $m_A$ is the multiplication of $A$ for $i\geq 1$.
\begin{remark}
In the sequel, we treat $V$ as the vector space $V$ or the homogeneous space $A_1$ of a Koszul algebra $A$ freely. So we write $m_A$ shortly for the liner map $m_A({\pi_A}_{\mid V}\otimes \id_A):V\otimes A\to A$ for convenience.
\end{remark}

As vector spaces, the Ext-algebra of $A$ is
$$
E(A)=\bigoplus_{i\geq0}\uHom_A(W_i\otimes A,k_A)\cong \bigoplus_{i\geq0} W_i^{*}.
$$

\begin{remark}
For Koszul algebras, researchers always use their Koszul dual as a common tool, which are also isomorphic to their Ext-algebras. In this paper, we use the language of Ext-algebras to study the Nakayama automorphisms of graded Ore extensions, since it is convenient to obtain some induced maps.
%For a Koszul algebra, it is well known its Ext-algebra is isomorphic to its Koszul dual $A^!$, which is a common tool to study Koszul algebras. However, we use the language of Ext-algebras to consider the Nakayama automorphisms of graded Ore extensions in this paper, since there
\end{remark}
%In this section, we fix basic notations and recall definitions for this paper.
%

\begin{definition}
A connected graded algebra $C$ is \emph{Artin-Schelter regular}  (\emph{AS-regular}, for short) of dimension $d$, if it  has finite global dimension $d$, $\dim\left(\uExt_{C}^{d}(k_C, C_C)\right)=1$ and $\dim\left(\uExt_{C}^{i}(k_C, C_C)\right)=0$ for $i\neq d$.
\end{definition}
AS-regular algebras have an important homological invariant.

\begin{theorem} \cite[Proposition 4.5 (b)]{YZ}\label{thm: Nakayama aut of AS regular algebras} If $C$ is AS-regular of dimension $d$, then there exists a unique graded automorphism $\mu_C$ of $C$ such that
$$\uExt^i_{C^e}(C,C^e)\cong\left\{\begin{array}{ll} 0&i\neq d,\\C^{\mu_C}(l)&i=d,   \end{array}\right. \quad \text{as graded $C$-bimodules},$$
where $C^e=C\otimes C^{op}$ is the enveloping algebra of $C$ for some $l\geq 0$.
\end{theorem}
The automorphism $\mu_C$ is called the \emph{Nakayama automorphism of $C$}. In particular, if the Nakayama automorphism of an AS-regular algebra is the identity map, then it is a Calabi-Yau (CY, for short) algebra (see \cite{Gin,RRZ1}).

The Ext-algebras of AS-regular algebras also carry important information. Let $E=\oplus_{i\in \mathbb{Z}}E^i$ be a finite dimensional graded algebra. We say $E$ is a \emph{graded Frobenius algebra}, if there exists a nondegenerate associative graded bilinear form $\langle -,-\rangle: E\otimes E\to k(d)$ for some $d\in \mathbb{Z}$. In particular, there exits a graded automorphism $\mu_E$ of $E$ such that
$$
\langle \alpha,\beta\rangle=(-1)^{ij}\langle \beta,\mu_E(\alpha)\rangle,
$$
for any $\alpha\in E^i,\beta\in E^j$. We call the automorphism $\mu_E$ is the (classical) \emph{Nakayama automorphism} of the graded Frobenius algebra $E$. (see \cite{Sm} for details).

%\begin{theorem}\label{thm: Ext of AS is Frobenius}
%\cite[Corollary D]{LPWZ4} The algebra $C$ is AS-regular if and only if the Ext-algebra $E(C)$ is a graded Frobenius algebra.
%\end{theorem}

In this paper, we mainly consider Koszul AS-regular algebras. We list some important results about Koszul AS-regular algebras below.

\begin{theorem}\label{thm: properties of Koszul regular algebras}Let $A$ be a Koszul AS-regular algebra of dimension $d$ and $\mu_A$ the Nakayama automorphism of $A$.
\begin{enumerate}
\item \cite[Proposition 5.10]{Sm} The Ext-algebra $E(A)$ of $A$ is a graded Frobenius algebra.
\item\cite[Proposition 3]{V} Let $\mu_E$ be the (classical) Nakayama automorphism of $E(A)$, then
    %$
%    {\mu_E}=({\mu_A})^{\!}.
%    $
    $$
    {\mu_E}_{|E^1(A)}=({\mu_A}_{|A_1})^*,
    $$
    where $E^1(A)$ is identified with $V^*=A_1^*$.
\item \cite[Lemma 4.3]{HVZ}
Any nonzero element $\omega\in W_d$ is a ${\mu_A}_{\mid V}$-twisted superpotential.

\item \cite[Theorem 11]{DV}\cite[Theorem 4.4(i)]{HVZ} For any nonzero element $\omega\in W_d$, $A\cong \mathcal{A}(\omega,d-2)$.
\end{enumerate}
\end{theorem}
\begin{remark}
If $A$ is a Koszul AS-regular algebra of dimension $d$, then $\dim W_d=1$. So an element $\omega\in W_d$ is nonzero is equivalent to it is a basis of $W_d$.
\end{remark}
%We give a brief introduction for graded Frobenius algebras.
%
%
%

%\cite[Theorem 4.2]{RRZ2} If $A$ is noetherian AS-regular of type $(d,l)$ generated by degree $1$. Let $\mu_A$ and $\mu_E$ be Nakayama automorphisms of $A$ and $E(A)$, respectively. Then
%    $$
%    {\mu_E}_{|E^1(A)}=({\mu_A}_{|A_1})^*,
%    $$
%    where $E^1(A)$ is identified with $A_1^*$.

%

\section{A sequence pair}
Let $V$ be a vector space with a basis $\{x_1,x_2,\cdots,x_n\}$, $T(V)$ the tensor algebra over $V$, and $A=T(V)/(R)$ a Koszul algebra, where $R$ is a subspace of $V^{\otimes 2}$. Recall the Koszul complex as follows
\begin{equation}\label{resolution of k_A}
\cdots\xlongrightarrow{\partial_{d+1}^A} W_d\otimes A\xlongrightarrow{\partial^A_d}W_{d-1}\otimes A\xlongrightarrow{\partial^A_{d-1}}\cdots\xlongrightarrow{\partial^A_{3}}W_2\otimes A\xlongrightarrow{\partial^A_{2}}W_1\otimes A\xlongrightarrow{\partial^A_{1}}A\xlongrightarrow{\varepsilon_A}k_A\to0,
\end{equation}
where $W_0=k$, $W_1=V$, $W_2=R$ and $
W_{i}=\bigcap_{0\leq s\leq i-2} V^{\otimes s}\otimes R\otimes V^{\otimes i-s-2}
$ for $i\geq 3$.

Let $\overline{\sigma}$ be a graded automorphism of $A$ and $\overline{\delta}$ a degree-one $\overline{\sigma}$-derivation of $A$. Then we have a graded Ore extension $B=[z;\overline{\sigma},\overline{\delta}]$ with $\deg z=1$. Clearly, $B$ is a quotient algebra of the tensor algebra $T(V\oplus k\{z\})$ and a Koszul algebra. Write $\pi_B$ for the canonical projection from $T(V\oplus k\{z\})$ to $B$.

Write $\sigma=\overline{\sigma}_{\mid V}\in GL(V)$. It is easy to know that $\sigma^{\otimes i}(W_i)=W_i$ for each $i\geq 0$  and $\sigma_{T}:=\oplus_{i=0}^{\infty}\sigma^{\otimes i}$ is a graded automorphism of $T(V)$ such that $\overline{\sigma}$ is induced by $\sigma_{T}$.

Choose a linear map $\delta:V\to V\otimes V$ such that $\pi_A\delta=\overline{\delta}_{\mid V}$. In fact, $\delta$ extends to a degree-one $\sigma_T$-derivation of $T(V)$ (also denoted by $\delta$) in a unique way, and
\begin{equation}\label{condition for delta}
\delta(R)\subseteq R\otimes V+V\otimes R.
\end{equation}
So $\overline{\delta}$ can be induced by $\delta$. All $\overline{\sigma}$-derivations can be obtained in this way.

The condition (\ref{condition for delta}) provides an approach to decomposing $\delta$ into two parts, that is, there exist two linear maps
$\delta_{2,r}: R\to R\otimes V,\delta_{2,l}: R\to V\otimes R.
$ such that
\begin{equation}\label{definition of delta_2}
\delta_{\mid R}=\delta_{2,r}+\delta_{2,l}.
\end{equation}
%Although $B$ is defined by $\overline{\sigma}$ and $\overline{\delta}$, all of our procedure  work on the tensor algebra $T(V)$ and the $\sigma_T$-derivation $\delta$. To handle $\delta$,
The decomposition can be realized as follows. One can choose a basis $\{r_1,r_2,\cdots,r_t\}$ of $R$, obtain that
\begin{equation}\label{decomposition of delta}
\delta(r_i)=\sum_{j=1}^t r_j\otimes \alpha^{i}_{j}+\sum_{j=1}^t \beta^{i}_{j}\otimes r_j\in R\otimes V+V\otimes R,
\end{equation}
for some $\alpha^{i}_j,\beta^{i}_j\in V$ and $j,i=1,\cdots,t$, and then define
\begin{equation*}%
\qquad\delta_{2,r}(r_i)=\sum_{j=1}^t r_j\otimes \alpha^{i}_{j},\qquad
\delta_{2,l}(r_i)=\sum_{j=1}^t \beta^{i}_{j}\otimes r_j, \qquad \forall i=1,\cdots,t.
\end{equation*}

\begin{remark}
It is clear that the choice of $\delta_{2,r}$ and $\delta_{2,l}$ for $\delta$ is not unique, which depends on the decomposition (\ref{decomposition of delta}).
\end{remark}

\subsection{Minimal free resolutions}

Now we begin to construct a minimal free resolution of $k_B$. In the sequel, write $\delta_{1,r}=\delta_{1,l}=\delta_{|V}$, and  $\lambda_z$ (resp. $\rho_z$) for the left (resp. right) multiplication of $z$ on $B$.

Applying $-\otimes_AB$ to (\ref{resolution of k_A}), one obtains an exact sequence,
$$
\cdots\xlongrightarrow{\partial_{d+1}} W_d\otimes B\xlongrightarrow{\partial_d}W_{d-1}\otimes B\xlongrightarrow{\partial_{d-1}}\cdots\xlongrightarrow{\partial_{3}}W_2\otimes B\xlongrightarrow{\partial_{2}}W_1\otimes B\xlongrightarrow{\partial_{1}}B\xlongrightarrow{\varepsilon_A\otimes_AB}B/A_{\geq1}B\to0,
$$
where $\partial_{i}=\partial^A_{i}\otimes_AB=(\id_V^{\otimes i-1}\otimes m_B)(\id_V^{\otimes i-1}\otimes\pi_B\otimes\id_B)=\id_V^{\otimes i-1}\otimes m_B$ and $m_B$ is the multiplication of $B$ for $i\geq1$.

\begin{lemma}\label{lemma: construction of delta_r}
There exist graded linear maps $\delta_{i,r}:W_i\to W_i\otimes V$ for $i\geq 2$ such that the following diagram is commutative
$${\Small
\xymatrix{
\cdots\ar[r] & W_d\otimes B(-1) \ar[r]^{\partial_d}\ar[d]^{\phi_d} &W_{d-1}\otimes B(-1)\ar[r]^(0.65){\partial_{d-1}}\ar[d]^{\phi_{d-1}}&\cdots\ar[r]^(0.35){\partial_{2}} &W_1\otimes B(-1)\ar[r]^(0.55){\partial_{1}}\ar[d]^{\phi_1} &B(-1)\ar[r]^(0.38){{\varepsilon_A\otimes_AB}}\ar[d]^{\phi_0}&B/A_{\geq1}B(-1)\ar[d]^{\lambda_z}\ar[r]&0
\\
\cdots\ar[r] & W_d\otimes B \ar[r]^{\partial_d} &W_{d-1}\otimes B\ar[r]^(0.6){{\partial_{d-1}}}&\cdots\ar[r]^(0.41){{\partial_{2}}} &W_1\otimes B\ar[r]^(0.6){{\partial_{1}}} &B\ar[r]^(0.38){{\varepsilon_A\otimes_AB}}&B/A_{\geq1}B\ar[r]&0,
}
}
$$
where graded right $B$-module homomorphisms
%\begin{align*}
%&\phi_0=z\cdot-, \\
%%&\phi_1(w_1\otimes b)=\sigma(w_1)\otimes zb+\delta(w_1)b,\\
%&\phi_{i}(w_i\otimes b)=\sigma^{\otimes i}(w_i)\otimes zb+\delta_{i,1}(w_i)b,
%\end{align*}
\begin{align*}
&\phi_0=\lambda_z, \\
%&\phi_1(w_1\otimes b)=\sigma(w_1)\otimes zb+\delta(w_1)b,\\
&\phi_{i}=\sigma^{\otimes i}\otimes \lambda_z+(\id_V^{\otimes i}\otimes m_B)\left(\delta_{i,r}\otimes\id_B\right), \quad i\geq 1.
\end{align*}
%Here, $\delta_{i,1}$ represents $(\id_V^{\otimes {i}}\otimes \pi_B)\delta_{i,1}$ for short.

Moreover, as linear maps $W_i\to W_{i-1}\otimes B$,
\begin{equation}\label{condition for delta_r}
(\id_V^{\otimes i-1} \otimes m_B)\delta_{i,r}=(\id_V^{\otimes i-1}\otimes m_B)\left(\sigma^{\otimes i-1}\otimes\delta+\delta_{i-1,r}\otimes \id_V\right),
\quad i\geq 2.
\end{equation}
\end{lemma}
\begin{proof}
It's easy to check that $\lambda_z(\varepsilon_A\otimes_A B)=(\varepsilon_A\otimes_AB)\phi_0$ and $\phi_0\partial_1=\partial_1\phi_1$, since
\begin{align*}
&\phi_0\partial_1=\lambda_zm_B=m_B(\sigma\otimes \lambda_z+(\pi_B\otimes\id_B)(\delta\otimes\id_B))=m_B(\sigma\otimes \lambda_z)+m_B(\overline{\delta}\otimes \id_B),\\
&\partial_1\phi_1=m_B(\sigma\otimes\lambda_z)+m_B(\id_V\otimes m_B)(\delta\otimes \id_B)=m_B(\sigma\otimes\lambda_z)+m_B(\overline{\delta}\otimes \id_B).
\end{align*}
By the construction of $\delta_{2,r}$ and $\delta_{2,l}$, we have
$$
(\id_V\otimes m_B)\left(\sigma\otimes\delta+\delta\otimes \id_V\right)=
(\id_V\otimes m_B)\delta=(\id_V\otimes m_B)(\delta_{2,r}+\delta_{2,l})=
(\id_V\otimes m_B)\delta_{2,r},
$$
in case restricted on $W_2=R$. One obtains
\begin{align*}
\partial_2\phi_2&=(\id_V\otimes m_B)(\sigma^{\otimes 2}\otimes \lambda_z)+(\id_V\otimes m_B)(\id_V^{\otimes2}\otimes m_B)(\delta_{2,r}\otimes \id_B)\\
&=(\id_V\otimes m_B)(\sigma^{\otimes 2}\otimes \lambda_z)+(\id_V\otimes m_B)(\id_V^{\otimes2}\otimes m_B)(\delta_{2,r}\otimes \id_B+\delta_{2,l}\otimes\id_B)\\
&=(\id_V\otimes m_B)(\sigma^{\otimes 2}\otimes \lambda_z)+(\id_V\otimes m_B)(\id_V^{\otimes2}\otimes m_B)(\delta \otimes\id_B),\\
\phi_1\partial_2&=(\sigma\otimes \lambda_z)(\id_V\otimes m_B)+(\id_V\otimes m_B)(\delta\otimes \id_B)(\id_V\otimes m_B)\\
&=(\id_V\otimes m_B)(\sigma^{\otimes 2}\otimes \lambda_z)+(\id_V\otimes m_B)(\id_V^{\otimes2}\otimes m_B)\left((\sigma\otimes\delta)\otimes\id_B+(\delta\otimes\id_V)\otimes \id_B\right)\\
&=(\id_V\otimes m_B)(\sigma^{\otimes 2}\otimes \lambda_z)+(\id_V\otimes m_B)(\id_V^{\otimes2}\otimes m_B)(\delta\otimes\id_B).
\end{align*}
So $\partial_2\phi_2=\phi_1\partial_2$. By the Comparison Theorem, there exists a graded $B$-module homomorphism $\phi_i: W_i\otimes B(-1)\to W_i\otimes B$ such that $\phi_i\partial_{i+1}=\partial_{i+1}\phi_{i+1}$  for any $i\geq 3$.

%Let $i\geq 3$. Suppose there exists a linear map $\delta_{j,1}:W_j\to W_j\otimes V$ such that
%$$\phi_{j}=\sigma^{\otimes j}\otimes \lambda_z+(\id_V^{\otimes j}\otimes m_B)(\delta_{j,1}\otimes\id_B),$$
%and $\phi_{j-1}\partial_{j}=\partial_{j}\phi_{j}$ for any $j< i$.

It is clear that
$(\phi_i-\sigma^{\otimes i}\otimes\lambda_z)(W_i\otimes B_0)\subseteq W_i\otimes V,$ so write $\delta_{i,r}$ for the linear map $\phi_i-\sigma^{\otimes i}\otimes\lambda_z$ restricted on $W_i$ for any $i\geq 3$. One obtains that
$$
\phi_i=\sigma^{\otimes i}\otimes \lambda_z+(\id_V^{\otimes i}\otimes m_B)(\delta_{i,r}\otimes\id_B),\quad i\geq 3.
$$
Then $\partial_i\phi_i=(\id_V^{\otimes i-1}\otimes m_B)(\sigma^{\otimes i}\otimes\lambda_z)+(\id_V^{\otimes i-1}\otimes m_B)(\id_V^{\otimes i}\otimes m_B)(\delta_{i,r}\otimes \id_B)$, and
\begin{align*}
\phi_{i-1}\partial_i&=(\sigma^{\otimes i-1}\otimes \lambda_z)(\id_V^{\otimes i-1}\otimes m_B)
+(\id_V^{\otimes i-1}\otimes m_B)\left(\delta_{i-1,r}\otimes\id_B\right)(\id_V^{\otimes i-1}\otimes m_B)\\
&=(\id_V^{\otimes i-1}\otimes m_B)(\sigma^{\otimes i}\otimes\lambda_z)+
(\id^{\otimes i-1}_V\otimes m_B)(\id^{\otimes i}_V\otimes m_B)(\sigma^{\otimes i-1}\otimes\delta\otimes\id_B)\\
&\quad +(\id^{\otimes i-1}_V\otimes m_B)(\id^{\otimes i}_V\otimes m_B)(\delta_{i-1,r}\otimes\id_V\otimes\id_B).
\end{align*}
Since $\partial_i\phi_{i}=\phi_{i-1}\partial_i$, as linear maps $W_i\to W_{i-1}\otimes B$,
\begin{equation*}
\qquad(\id_V^{\otimes i-1}\otimes m_B)\delta_{i,r}=(\id_V^{\otimes i-1}\otimes m_B)\left(\sigma^{\otimes i-1}\otimes\delta+\delta_{i-1,r}\otimes \id_V\right),\qquad \forall i\geq 3,\qedhere
\end{equation*}
\end{proof}

\begin{remark}\label{remark: choice of delta_r}
By the proof of Lemma \ref{lemma: construction of delta_r}, any
linear map from $W_i$ to $W_i\otimes V$ satisfying (\ref{condition for delta_r}) can be chosen to be $\delta_{i,r}$ for $i\geq 2$. So the sequence $\{\delta_{i,r}\mid i\geq1\}$ of linear maps is not unique for the map $\overline{\delta}$, even when $\delta_{2,r}$ is fixed.
\end{remark}

By \cite[Theorem 1]{GS} or \cite[Lemma 2.4]{Phan}, one obtains a minimal free resolution of $k_B$.

\begin{lemma}\label{lemma: minimial resolution of k_B}
The following complex is exact:
\begin{equation}\label{resolution of k_B}
\begin{aligned}
\cdots\rightarrow
W_{d}\otimes& B(-1)\oplus W_{d+1}\otimes B
\xlongrightarrow{\left({\scriptscriptstyle\begin{array}{ll}{\scriptstyle -\partial_{d} }& {\scriptstyle0} \\ {\scriptstyle\phi_{d} }&{\scriptstyle \partial_{d+1}}\end{array}}\right)}
W_{d-1}\otimes B(-1)\oplus W_d\otimes B\xlongrightarrow{}\cdots\\
&\xlongrightarrow{}
W_1\otimes B(-1)\oplus W_{2}\otimes B \xlongrightarrow{\left({\scriptscriptstyle\begin{array}{ll}{\scriptstyle -\partial_{1} }& {\scriptstyle0} \\ {\scriptstyle\phi_1 }&{\scriptstyle \partial_{2}}\end{array}}\right)}
B(-1)\oplus W_1\otimes B\xlongrightarrow{\left(\begin{array}{ll}{\scriptstyle \phi_0}& {\scriptstyle \partial_1}\end{array}\right)}
B\xlongrightarrow{\varepsilon_B}
k_B\xlongrightarrow{} 0,
\end{aligned}
\end{equation}
for some sequence $\{\delta_{i,r}\}$ of linear maps as in Lemma \ref{lemma: construction of delta_r}.
\end{lemma}

%\begin{proposition}
%$$
%(\id_V^{\otimes i-1} \otimes m_B)\delta_{i,r}=(\id_V^{\otimes i-1} \otimes m_B)\delta
%$$
%\end{proposition}

%$
%\xymatrix{
%0\ar[r] & W_d\otimes B(-1) \ar[r]^{\begin{array}{l}-\partial_d \\ \phi_n\end{array}}
%&W_{d-1}\otimes B(-1)\ar[r]^(0.65){\partial_{d-1}}
%&\cdots\ar[r]^(0.35){\partial_{2}}
%&W_1\otimes B(-1)\ar[r]^(0.55){\partial_{1}}
%&B(-1)\ar[r]&
%k_B
%$

%$$
%\xymatrix{
%0\ar[r] & W_n\otimes B(-1) \ar[r]^{\overline{\partial_n}} &W_{n-1}\otimes B(-1)\ar[r]^{\overline{\partial_{n-1}}}&\cdots\ar[r]^{\overline{\partial_{3}}} &W_2\otimes B(-1)\ar[r]^{\overline{\partial_{2}}} &W_1\otimes B(-1)\ar[r]^{\overline{\partial_{1}}} &B(-1)\ar[r]&0\\
%0\ar[r] & W_n\otimes B \ar[r]^{\overline{\partial_n}} &W_{n-1}\otimes B\ar[r]^{\overline{\partial_{n-1}}}&\cdots\ar[r]^{\overline{\partial_{3}}} &W_2\otimes B\ar[r]^{\overline{\partial_{2}}} &W_1\otimes B\ar[r]^{\overline{\partial_{1}}} &B\ar[r]&0
%,
%}
%$$

\subsection{A sequence pair}

We have constructed a sequence $\{\delta_{i,r}\mid i\geq 2\}$ of linear maps from $\delta$ to obtain a minimal free resolution of $k_B$. It seems a right version of linear maps arose from $\delta$. Symmetrically, we construct a left version of linear maps. Write $\delta_{0,r}=\delta_{0,l}=0$.

\begin{lemma}\label{lemma: construction of delta_l}
Let $\{\delta_{i,r}:W_i\to W_i\otimes V\mid i\geq 1\}$ be a sequence of linear maps as in Lemma \ref{lemma: construction of delta_r}. Then there exists a unique set $\{\delta_{i,l}:W_i\to V\otimes W_i\mid i\geq 1\}$ of linear  maps with respect to $\{\delta_{i,r}\}$ such that
%$$
%-(\id_V^{\otimes i}\otimes m_B)
%\left((\delta_{i,l}-\delta_{i-1,l}\otimes \id_V)\otimes \id_B\right)=(\id_V^{\otimes i }\otimes m_B)\left((\delta_{i,r}-\sigma\otimes\delta_{i-1,r})\otimes\id_B\right)
%$$
%$$
%\delta_{i,l}-\delta_{i-1,l}\otimes \id_V=\delta_{i,r}-\sigma\otimes\delta_{i-1,r}
%$$
\begin{equation}\label{eq: relation between delta_l and delta_r}\qquad\quad
\delta_{i,r}+(-1)^i\delta_{i,l}=\sigma\otimes\delta_{i-1,r}+(-1)^i\delta_{i-1,l}\otimes \id_V,\qquad \forall i\geq 1.
\end{equation}
%Moreover, if $\{\delta'_{i,r}:W_i\to W_i\otimes V\mid i\geq 2\}$ is another set of linear maps as in Lemma \ref{lemma: construction of delta_r} and $\{\delta'_{i,l}:W_i\to V\otimes W_i\mid i\geq 2\}$ is another set satisfying (\ref{eq: relation between delta_l and delta_r}) with respect to $\{\delta'_{i,r}\}_{i=2}^{\infty}$, then $$\Image(\delta_{i,l}-\delta'_{i,l})\subseteq W_{i+1}.$$

\end{lemma}
\begin{proof}
Since $\delta_{1,r}=\delta_{1,l}$, the result holds if $i=1$.
By the definition (\ref{definition of delta_2}) of $\delta_{2,l}$, it is a linear map from  $W_2$ to $V\otimes W_2$. Clearly,
$$
\delta_{2,r}-\delta_{1,l}\otimes\id_V-\sigma\otimes  \delta_{1,r}=\delta_{2,r}-\delta\otimes\id_V-\sigma\otimes \delta=\delta_{2,r}-\delta=-\delta_{2,l}.
$$

Suppose $i\geq 3$ and there exist linear maps $\delta_{j,l}:W_j\to V\otimes W_j$ for  $j< i$, such that
$$
\delta_{j,r}+(-1)^j\delta_{j,l}=\sigma\otimes\delta_{j-1,r}+(-1)^j\delta_{j-1,l}\otimes \id_V.
$$
Then we have the following commutative diagram
$$
\xymatrix{
&W_i\otimes B(-1)  \ar@{-->}[ld]_{\overline{\theta}_i}\ar[d]^{\theta_i}&\\
V\otimes W_i\otimes B\ar[r]^(0.47){\id_V\otimes \partial_i} &V\otimes W_{i-1}\otimes B\ar[r]^{\id_V\otimes \partial_{i-1}} &V\otimes W_{i-2}\otimes B,
}
$$
where
$
\theta_i=(\id_V^{\otimes i }\otimes m_B)\left(\left((-1)^{i-1}\left(\delta_{i,r}-\sigma\otimes\delta_{i-1,r}\right)+\delta_{i-1,l}\otimes \id_V\right)\otimes\id_B\right).
$
In fact,
{\small
\begin{align*}
(-1)^{i-1}(\id_V \otimes \partial_{i-1})\theta_i&=(\id_V^{\otimes i-1} \otimes m_B)(\id_V^{\otimes i }\otimes m_B)\left((\delta_{i,r}-\sigma\otimes\delta_{i-1,r}+(-1)^{i-1}\delta_{i-1,l}\otimes \id_V)\otimes\id_B\right)\\
&=(\id_V^{\otimes i-1} \otimes m_B)(\id_V^{\otimes i }\otimes m_B)\left((\sigma^{\otimes i-1}\otimes \delta+\delta_{i-1,r}\otimes \id_V-\sigma\otimes\delta_{i-1,r}+(-1)^{i-1}\delta_{i-1,l}\otimes \id_V)\otimes\id_B\right)\\
&=(\id_V^{\otimes i-1} \otimes m_B)(\id_V^{\otimes i }\otimes m_B)\left(
(
\sigma^{\otimes i-1}\otimes \delta+(\delta_{i-1,r}+(-1)^{i-1}\delta_{i-1,l})\otimes\id_V-\sigma\otimes\delta_{i-1,r})
\otimes \id_B
\right)\\
&=
(\id_V^{\otimes i-1} \otimes m_B)(\id_V^{\otimes i }\otimes m_B)\left(
(
\sigma^{\otimes i-1}\otimes \delta+\sigma\otimes\delta_{i-2,r}\otimes\id_V-\sigma\otimes\delta_{i-1,r})
\otimes \id_B
\right)\\
&\quad+(-1)^{i-1}(\id_V^{\otimes i-1} \otimes m_B)(\id_V^{\otimes i }\otimes m_B)\left((\delta_{i-2,l}\otimes\id_V\otimes\id_V)\otimes\id_B\right)\\
&=
(\id_V^{\otimes i-1} \otimes m_B)(\id_V^{\otimes i }\otimes m_B)\left(
\left(
\sigma^{\otimes i-1}\otimes \delta+\sigma\otimes(\delta_{i-2,r}\otimes\id_V-\delta_{i-1,r})\right)
\otimes \id_B
\right)\\
&=
(\id_V^{\otimes i-1} \otimes m_B)(\id_V^{\otimes i }\otimes m_B)\left(
\left(
\sigma^{\otimes i-1}\otimes \delta-\sigma^{\otimes i-1}\otimes \delta)\right)
\otimes \id_B
\right)\\
&=0,
\end{align*}}
the second and sixth equalities hold by Lemma \ref{lemma: construction of delta_r}, the fourth equality holds by the assumption and the fifth equality holds by $W_i\subseteq V^{\otimes i-2}\otimes R$. So $\Image \theta_i\subseteq\Ker (\id_V\otimes\partial_{i-1})=\Image (\id_V\otimes\partial_{i})$, and there exists a graded $B$-module homomorphism $\overline{\theta}_{i}: W_i\otimes B(-1)\to V\otimes W_i\otimes B$ such that $(\id_V\otimes \partial_i)\overline{\theta}_i=\theta_i$, since $W_i\otimes B$ is free.

It is easy to know that ${\overline{\theta}_{i}}_{\mid W_i}$ is a map from $W_i$ to $V\otimes W_i$ by an argument on degree, denoted this map by $\delta_{i,l}$.
%Let $\delta_{i,l}$ be the restriction map of $\overline{\theta}_i$ on $W_i$.
Hence,
$$
\delta_{i,r}+(-1)^i\delta_{i,l}=\sigma\otimes\delta_{i-1,r}+(-1)^i\delta_{i-1,l}\otimes \id_V.
$$
%$${\small
%\xymatrix{
%\cdots\ar[r] & W_d\otimes B \ar[r]^(0.6){\partial_d}\ar[d]^{\theta_{d-1}} &W_{d-1}\otimes B\ar[r]^{\partial_{d-1}}\ar[d]^{\theta_{d-2}} &\cdots\ar[r]^(0.3){\partial_{3}}&W_2\otimes B\ar[r]^(0.47){\partial_{2}}\ar[d]^{\theta_1} &W_1\otimes B\ar[d]^{\theta_0}\ar[rd]^{x_i^*\otimes \varepsilon_B}
%\\
%\cdots\ar[r] &W_{d-1}\otimes B(-1)\ar[r]^(0.47){{-\partial_{d-1}}}&W_{d-2}\otimes B(-1)\ar[r]^(0.6){-\partial_{d-2}} &\cdots\ar[r]^(0.37){-{\partial_{2}}} &W_1\otimes B(-1)\ar[r]^(0.55){{-\partial_{1}}}
%&B(-1)\ar[r]^{\varepsilon_B}&
%k_B(-1),
%}}
%$$
The uniqueness can be obtained by the construction easily.
%
%We show the last consequence by induction. Obviously, $\delta_{\mid W^{\otimes 2}}=\delta_{2,r}-\delta_{2,l}=\delta'_{2,r}-\delta'_{2,l},$ so
%$$
%\delta_{2,r}-\delta'_{2,r}=\delta_{2,l}-\delta'_{2,l},
%$$
%and $\Image (\delta_{2,l}-\delta'_{2,l})\subseteq (W_2\otimes V)\cap (V\otimes W_2)=W_3$.
%
%Suppose $\Image(\delta_{j,l}-\delta'_{j,l})\subseteq W_{j+1}$ for any $j<i$, where $i\geq 3$.
%\begin{align*}
%(\id_V\otimes m_B)(\delta_{i,l}-\delta'_{i,l})
%\theta_i&=(\id_V^{\otimes i-1} \otimes m_B)(\id_V^{\otimes i }\otimes m_B)\left((\delta_{i,r}-\sigma\otimes\delta_{i-1,r}-\delta_{i-1,l}\otimes \id_V)\otimes\id_B\right)\\
%&=(\id_V^{\otimes i-1} \otimes m_B)(\id_V^{\otimes i }\otimes m_B)\left((\sigma^{\otimes i-1}\otimes \delta+\delta_{i-1,r}\otimes \id_V-\sigma\otimes\delta_{i-1,r}-\delta_{i-1,l}\otimes \id_V)\otimes\id_B\right)\\
%\end{align*}
\end{proof}

\begin{definition} Let $\{\delta_{i,r}:W_i\to W_i\otimes V\mid i\geq 1\}$ be a sequence of linear maps constructed in Lemma \ref{lemma: construction of delta_r}, and  $\{\delta_{i,l}:W_i\to V\otimes W_i \mid i\geq 1\}$ a sequence of linear maps constructed in Lemma \ref{lemma: construction of delta_l}. Then
$\left(\{\delta_{i,r}\},\{\delta_{i,l}\}\right)$ is called a \emph{sequence pair} for the $\overline{\sigma}$-derivation $\overline{\delta}$.
\end{definition}

We give a relation between $\{\delta_{i,r}\}$ and $\{\delta_{i,l}\}$, which will be useful in the construction of twisted superpotentials for graded Ore extensions.

\begin{proposition}\label{prop: relation between delta_r and delta_l}
Let $\left(\{\delta_{i,r}\},\{\delta_{i,l}\}\right)$ be a sequence pair for $\overline{\delta}$. Then
\begin{enumerate}
\item For any integer $d\geq 1$,
$$
\sum_{i=1}^d (-1)^{i}(\delta_{i,r}\otimes\id_V^{\otimes d-i})=(-1)^{d+1}\sum_{i=1}^{d}(-1)^{i}(\sigma^{\otimes d-i}\otimes\delta_{i,l}).
$$
\item For any integer $d\geq 2$, $$
\sum_{i=1}^{d-1}(-1)^{i+1}(\sigma^{\otimes d-i-1}\otimes \delta_{i,l}\otimes\id_V)=(-1)^{d+1}\sum_{i=1}^{d-1}(-1)^i(\delta_{i,r}\otimes\id_V^{\otimes d-i}).
$$
\end{enumerate}

\end{proposition}

\begin{proof} (a) If $d=1$, the equality is clear. Assume $d\geq 2$. By Lemma \ref{lemma: construction of delta_l},
\begin{align*}
&\sum_{i=j+1}^d (-1)^{i}(\sigma^{\otimes j}\otimes\delta_{i-j,r}\otimes\id_V^{\otimes d-i})\\
=&\sum_{i=j+1}^d (-1)^{i}(\sigma^{\otimes j}\otimes\delta_{i-j,r}\otimes\id_V^{\otimes d-i})+(-1)^j\sum_{i=j+1}^d \sigma^{\otimes j}\otimes\delta_{i-j,l}\otimes\id_V^{\otimes d-i}-(-1)^j\sum_{i=j+1}^d \sigma^{\otimes j}\otimes\delta_{i-j,l}\otimes\id_V^{\otimes d-i}\\
=&\sum_{i=j+1}^d (-1)^{i}\left(\sigma^{\otimes j}\otimes(\delta_{i-j,r}+(-1)^{i-j}\delta_{i-j,l})\otimes\id_V^{\otimes d-i}\right)-(-1)^j\sum_{i=j+1}^d \sigma^{\otimes j}\otimes\delta_{i-j,l}\otimes\id_V^{\otimes d-i}\\
=&\sum_{i=j+2}^d (-1)^{i}(\sigma^{\otimes j+1}\otimes\delta_{i-j-1,r}\otimes\id_V^{\otimes d-i})+(-1)^j\sum_{i=j+2}^d \sigma^{\otimes j}\otimes\delta_{i-j-1,l}\otimes\id_V^{\otimes d-i+1}-(-1)^j\sum_{i=j+1}^d \sigma^{\otimes j}\otimes\delta_{i-j,l}\otimes\id_V^{\otimes d-i}\\
=&\sum_{i=j+2}^d (-1)^{i}(\sigma^{\otimes j+1}\otimes\delta_{i-j-1,r}\otimes\id_V^{\otimes d-i})-(-1)^j\sigma^{\otimes j}\otimes\delta_{d-j,l},
\end{align*}
for any $j=0,1,\cdots,d-2$. Then
\begin{align*}
&\sum_{j=0}^{d-2}\sum_{i=j+1}^d (-1)^{i}(\sigma^{\otimes j}\otimes\delta_{i-j,r}\otimes\id_V^{\otimes d-i})=\sum_{j=0}^{d-2}\sum_{i=j+2}^d (-1)^{i}(\sigma^{\otimes j+1}\otimes\delta_{i-j-1,r}\otimes\id_V^{\otimes d-i})-\sum_{j=0}^{d-2}(-1)^j\sigma^{\otimes j}\otimes\delta_{d-j,l},\\
&\sum_{i=1}^d (-1)^{i}(\delta_{i,r}\otimes\id_V^{\otimes d-i})=(-1)^{d}(\sigma^{\otimes d-1}\otimes\delta_{1,r})-\sum_{j=0}^{d-2}(-1)^j\sigma^{\otimes j}\otimes\delta_{d-j,l}=-\sum_{j=0}^{d-1}(-1)^j\sigma^{\otimes j}\otimes\delta_{d-j,l}.
\end{align*}
The result follows.

(b) By Lemma \ref{lemma: construction of delta_l},
$$
\sum_{i=1}^{d-1}(-1)^{i+1}\sigma^{\otimes d-i-1}\otimes \delta_{i,l}\otimes\id_V+\sum_{i=1}^{d-1}\sigma^{\otimes d-i}\otimes\delta_{i,r}
=\sum_{i=1}^{d-1}\sigma^{{\otimes d-i-1}}\otimes \delta_{i+1,r}+\sum_{i=1}^{d-1}(-1)^{i+1}\sigma^{\otimes d-i-1}\otimes\delta_{i+1,l}.
$$
Then
\begin{align*}
\sum_{i=1}^{d-1}(-1)^{i+1}\sigma^{\otimes d-i-1}\otimes \delta_{i,l}\otimes\id_V&=\delta_{d,r}-\sigma^{\otimes d-1}
\otimes \delta_{1,r}+\sum_{i=2}^{d}(-1)^{i}\sigma^{\otimes d-i}\otimes\delta_{i,l}\\
&=\delta_{d,r}-\sigma^{\otimes d-1}
\otimes \delta_{1,r}+\sigma^{\otimes d-1}\otimes\delta_{1,l}-(-1)^{d}\sum_{i=1}^d(-1)^i(\delta_{i,r}\otimes\id_V^{\otimes d-i})\\
&=(-1)^{d+1}\sum_{i=1}^{d-1}(-1)^i(\delta_{i,r}\otimes\id_V^{\otimes d-i}),
\end{align*}
where the second equality holds by (a).
\end{proof}

As shown above, a sequence pair $\left(\{\delta_{i,r}\},\{\delta_{i,l}\}\right)$ is constructed from a linear map $\delta: V\to V\otimes V$ which induces the $\overline{\sigma}$-derivation $\overline{\delta}$ in the graded Ore extensions $B$. It is also shown that such sequence pairs vary according to the decomposition (\ref{decomposition of delta}) and the choices of $\{\delta_{i,r}\}$ in Lemma \ref{lemma: construction of delta_r}. On the other hand, the $\overline{\sigma}$-derivation $\overline{\delta}$ can be  induced from different linear maps $\delta,\delta':V\to V\otimes V$, which also arise different sequence pairs for $\overline{\delta}$. Here, we give a relation between different sequence pairs.

\begin{proposition}\label{prop: relation betwenn different delta_r} Let $\delta,\delta'$ be two linear maps from $V$ to $V\otimes V$ such that $\pi_A\delta=\pi_A\delta'=\overline{\delta}_{\mid V}$, and $\left(\{\delta_{i,r}\},\{\delta_{i,l}\}\right)$ and $\left(\{\delta'_{i,r}\},\{\delta'_{i,l}\}\right)$ two sequence pairs for $\overline{\delta}$ constructed from $\delta$ and $\delta'$ respectively. Then for any $i\geq1$,
\begin{enumerate}
\item $
\Image\left(\sum_{j=0}^{i-1}(-1)^j\left((\delta_{i-j,r}-\delta'_{i-j,r})\otimes\id_V^{\otimes j}\right)\right)\subseteq W_{i+1}.
$
\item $
\Image\left((\delta_{i,l}-\delta'_{i,l})+(-1)^{i}\sum_{j=1}^{i-1}(-1)^j\left(\sigma\otimes(\delta_{i-j,r}-\delta'_{i-j,r})\otimes\id_V^{\otimes j-1}\right)\right)\subseteq W_{i+1}.
$
\end{enumerate}
%where $\delta=\delta_{1,r}=\delta_{1,l},\delta'=\delta'_{1,r}=\delta'_{1,l}.$
Moreover, if $W_{d+1}=0$ for some $d\geq1$, then
\begin{align*}
\sum_{j=0}^{d-1}(-1)^{j}(\delta_{d-j,r}\otimes\id_V^{\otimes j})&=\sum_{j=0}^{d-1}(-1)^{j}(\delta'_{d-j,r}\otimes\id_V^{\otimes j}),\\
\sum_{j=0}^{d-1}(-1)^{j}(\sigma^{\otimes j}\otimes \delta_{d-j,l})&=\sum_{j=0}^{d-1}(-1)^{j}(\sigma^{\otimes j}\otimes \delta'_{d-j,l}),\\
%\sum_{i=1}^d (-1)^{d-i}(\sigma^{\otimes d-i}\otimes \delta_{i,l})(\omega)&=\sum_{i=1}^d (-1)^{d-i}(\sigma^{\otimes d-i}\otimes \delta'_{i,l})(\omega),\\
(-1)^{d}\delta_{d,l}+\sum_{j=1}^{d-1}(-1)^{j}\left(\sigma\otimes \delta_{d-j,r}\otimes\id_V^{\otimes j-1}\right)&=(-1)^{d}\delta'_{d,l}+\sum_{j=1}^{d-1}(-1)^{j}\left(\sigma\otimes \delta'_{d-j,r}\otimes\id_V^{\otimes j-1}\right).
\end{align*}
%
%
%
%
%\begin{align*}
%\delta_{d,r}-\delta'_{d,r}&=\sum_{j=1}^{d-1}(-1)^{j+1}\left((\delta_{d-j,r}-\delta'_{d-j,r})\otimes\id_V^{\otimes j}\right),\\
%\delta_{d,l}-\delta'_{d,l}&=(-1)^{d-1}\sum_{j=1}^{d-1}(-1)^{j}\left(\sigma\otimes (\delta_{d-j,r}-\delta'_{d-j,r})\otimes\id_V^{\otimes j-1}\right).
%\end{align*}

%If $W_{d+1}=0$, then $\sum_{j=0}^{i-2}(-1)^j\left(\delta_{d-j,r}\otimes\id_V^{\otimes j}\right)=\sum_{j=0}^{i-2}(-1)^j\left(\delta'_{d-j,r}\otimes\id_V^{\otimes j}\right)$
\end{proposition}
\begin{proof}
(a)
Let $\{\phi_{i}\}$ and $\{\phi'_{i}\}$ be two lifts of $\lambda_z:B/A_{\geq 1}B(-1)\to B/A_{\geq1}B$ as in Lemma \ref{lemma: construction of delta_r} associated with $\{\delta_{i,r}\}$ and $\{\delta'_{i,r}\}$ respectively. Clearly, $\phi_0=\phi'_0$, %$\phi_1=\phi'_1$
and $\phi_i-\phi'_i= (\id_V^{\otimes i}\otimes m_B)\left((\delta_{i,r}-\delta'_{i,r})\otimes\id_B\right)$ for any $i\geq 1$. There exist graded $B$-module homomorphisms $s_i:W_i\otimes B(-1)\to W_{i+1}\otimes B$ for $i\geq 1$ such that the following diagram is commutative
$${\Small
\xymatrix{
\cdots\ar[r] & W_i\otimes B(-1) \ar[r]^{\partial_d}\ar[ld]_{s_i}\ar[d]^{\phi_i-\phi'_i} &W_{i-1}\otimes B(-1)\ar[r]^(0.65){\partial_{d-1}}\ar[ld]_{s_{i-1}}\ar[d]^{\phi_{i-1}-\phi'_{i-1}}&\cdots\ar[r]^(0.35){\partial_{3}}&W_2\otimes B(-1)\ar[ld]_{s_2}\ar[r]^(0.55){\partial_2}\ar[d]^{\phi_2-\phi'_2} &W_1\otimes B(-1)\ar[ld]_{s_1}\ar[r]^(0.55){\partial_{1}}\ar[d]^{\phi_1-\phi'_1} &B(-1)\ar[r]\ar[d]^{0}\ar[ld]_{0}&0
\\
\cdots\ar[r] & W_i\otimes B \ar[r]^{\partial_i} &W_{i-1}\otimes B\ar[r]^(0.6){{\partial_{d-1}}}&\cdots\ar[r]^(0.41){{\partial_{3}}}&W_2\otimes B\ar[r]^(0.55){\partial_{1}} &W_1\otimes B\ar[r]^(0.6){{\partial_{1}}} &B\ar[r]&0.
}
}
$$
that is, $\phi_i-\phi'_i=s_{i-1}\partial_i+\partial_{i+1}s_{i},$ for any $i\geq 1$, where $s_0=0$. By an easy argument on degree, one obtains that for any $i\geq 1$, $\Image {s_i}_{\mid W_i}\subseteq W_{i+1}$ and
$${s_i}_{\mid W_i}=\delta_{i,r}-\delta'_{i,r}-{s_{i-1}}_{\mid W_{i-1}}\otimes\id_V.$$

Since ${s_{1}}_{\mid W_1}=\delta_{1,r}-\delta'_{1,r}$, we have
$$
{s_i}_{\mid W_i}=\sum_{j=0}^{i-1}(-1)^j\left((\delta_{i-j,r}-\delta'_{i-j,r})\otimes\id_V^{\otimes j}\right).
$$
Then the result follows.

%Suppose

(b)  By Lemma \ref{lemma: construction of delta_l}, we have
\begin{align*}
\delta_{j,r}+(-1)^j\delta_{j,l}=\sigma\otimes\delta_{j-1,r}+(-1)^j\delta_{j-1,l}\otimes \id_V,\qquad
\delta'_{j,r}+(-1)^j\delta'_{j,l}=\sigma\otimes\delta'_{j-1,r}+(-1)^j\delta'_{j-1,l}\otimes \id_V.
\end{align*}
for any $j\geq 1$. So
$$
(\delta_{j,r}-\delta'_{j,r})\otimes\id_V^{\otimes i-j}
-\sigma\otimes(\delta_{j-1,r}-\delta'_{j-1,r})\otimes\id_V^{\otimes i-j}
=(-1)^j(\delta_{j-1,l}-\delta'_{j-1,l})\otimes \id_V^{\otimes i-j+1}-(-1)^j(\delta_{j,l}-\delta'_{j,l})\otimes\id_V^{\otimes i-j},
$$
for any $1\leq j\leq i$, where $\delta_{0,r}=\delta_{0,l}=\delta'_{0,r}=\delta'_{0,l}=0$. Then
\begin{align*}
&\sum_{j=1}^i(-1)^{i-j}(\delta_{j,r}-\delta'_{j,r})\otimes\id_V^{\otimes i-j}
-\sum_{j=1}^i(-1)^{i-j}\sigma\otimes(\delta_{j-1,r}-\delta'_{j-1,r})\otimes\id_V^{\otimes i-j}\\
=&\sum_{j=1}^i(-1)^i(\delta_{j-1,l}-\delta'_{j-1,l})\otimes \id_V^{\otimes i-j+1}-\sum_{j=1}^i(-1)^i(\delta_{j,l}-\delta'_{j,l})\otimes\id_V^{\otimes i-j}\\
=&(-1)^{i+1}(\delta_{i,l}-\delta'_{i,l}).
\end{align*}
Equivalently,
$$
\sum_{j=0}^{i-1}(-1)^{j}(\delta_{i-j,r}-\delta'_{i-j,r})\otimes \id_V^{\otimes j}=(-1)^{i+1}(\delta_{i,l}-\delta'_{i,l})-\sum_{j=1}^{i-1}(-1)^{j}\sigma\otimes(\delta_{i-j,r}-\delta'_{i-j,r})\otimes \id_V^{\otimes j-1}.
$$
Then the results holds by (a).

The last consequence is an immediate consequence of (a, b) and Proposition \ref{prop: relation between delta_r and delta_l}(a).
\end{proof}
\section{Nakayama automorphisms of graded Ore extensions of Koszul AS-regular algebras}
Keep the notations in the last section. In this section, we always assume $A=T(V)/(R)$ is a Koszul AS-regular algebra of dimension $d$, where $V$ is a vector space with a basis $\{x_1,x_2,\cdots,x_n\}$. Write $\mu_A$ for the Nakayama automorphism of $A$. In this case, by \cite[Proposition 3.1.4]{SZ}, one obtains
$$
\dim W_i=
\left\{
\begin{array}{ll}
\dim W_{d-i} & \text{if } 0\leq i\leq d.\\
0            & \text{if } i>d.
\end{array}
\right.
$$
Write $\{\eta_1,\eta_2,\cdots,\eta_n\}$ for a basis of $W_{d-1}$ and  $\omega$ for a basis of $W_d$.

It is well known that the graded Ore extension $B=A[z;\overline{\sigma},\overline{\delta}]$ is a Koszul AS-regular algebra of dimension $d+1$ provided $\overline{\sigma}$ is a graded automorphism of $A$. Write $\sigma=\overline{\sigma}_{\mid V}$. The minimal free resolution (\ref{resolution of k_B}) of $k_B$ becomes
\begin{equation}\label{resolution of koszul regular}
\begin{aligned}
0\xlongrightarrow{}
W_{d}\otimes& B(-1)
\xlongrightarrow{\left({\scriptscriptstyle\begin{array}{ll}{\scriptstyle -\partial_{d} }\\ {\scriptstyle\phi_{d} }\end{array}}\right)}
W_{d-1}\otimes B(-1)\oplus W_d\otimes B\xlongrightarrow{}\cdots\\
&\xlongrightarrow{}
W_1\otimes B(-1)\oplus W_{2}\otimes B \xlongrightarrow{\left({\scriptscriptstyle\begin{array}{ll}{\scriptstyle -\partial_{1} }& {\scriptstyle0} \\ {\scriptstyle\phi_1 }&{\scriptstyle \partial_{2}}\end{array}}\right)}
B(-1)\oplus W_1\otimes B\xlongrightarrow{\left(\begin{array}{ll}{\scriptstyle \phi_0}& {\scriptstyle \partial_1}\end{array}\right)}
B\xlongrightarrow{\varepsilon_B}
k_B\xlongrightarrow{} 0,
\end{aligned}
\end{equation}
for some sequence $\{\delta_{i,r}\}$ of linear maps as in Lemma \ref{lemma: construction of delta_r}.

\subsection{An invariant}

Let $\left(\{\delta_{i,r}\},\{\delta_{i,l}\}\right)$ be a sequence pair for $\overline{\delta}$. Since $\dim W_d=1$, there exists a unique pair $(\delta_r,\delta_l)$ of elements in $V$ with respect to $\left(\{\delta_{i,r}\},\{\delta_{i,l}\}\right)$  such that
$$\delta_{d,r}(\omega)=\omega\otimes \delta_r,\quad \delta_{d,l}(\omega)=\delta_l\otimes \omega.$$
%So there is a pair $(\delta_r,\delta_l)$ of elements in $V$ with respect to a sequence pair $\left(\{\delta_{i,r}\},\{\delta_{i,l}\}\right)$.

\begin{corollary}\label{coro: relation between delta_r and delta l}
Let $(\delta_r,\delta_l)$ and $(\delta'_r,\delta'_l)$ be two pairs with respect to two sequence pairs $\left(\{\delta_{i,r}\},\{\delta_{i,l}\}\right)$ and $\left(\{\delta'_{i,r}\},\{\delta'_{i,l}\}\right)$ for $\overline{\delta}$, respectively. Then
$$
\delta_r+{\mu_A}\sigma^{-1}(\delta_l)=\delta'_r+\mu_A\sigma^{-1}(\delta'_l).
$$
%where $\mu_A$ represents the linear map ${\mu_A}_{\mid V}$.
\end{corollary}
\begin{proof}Firstly, one obtains that
\begin{align*}
\tau_d^{d-1}(\mu_A\sigma^{-1}\otimes\id_V^{\otimes d})(\delta_{d,l}-\delta'_{d,l})(\omega) &= \omega\otimes\left(\mu_A\sigma^{-1}(\delta_l-\delta'_l)\right).
\end{align*}
On the other hand,
\begin{align*}
&\tau_d^{d-1}(\mu_A\sigma^{-1}\otimes\id_V^{\otimes d})\left((-1)^{d-1}\sum_{j=1}^{d-1}(-1)^{j}\left(\sigma\otimes (\delta_{d-j,r}-\delta'_{d-j,r})\otimes\id_V^{\otimes j-1}\right)(\omega)\right)\\
=&\tau_d^{d-1}\left((-1)^{d-1}\sum_{j=1}^{d-1}(-1)^{j}\left(\mu_A\otimes (\delta_{d-j,r}-\delta'_{d-j,r})\otimes\id_V^{\otimes j-1}\right)(\omega)\right)\\
=&\sum_{j=1}^{d-1}(-1)^{j}\left( (\delta_{d-j,r}-\delta'_{d-j,r})\otimes\id_V^{\otimes j}\right)\left((-1)^{d-1}\tau_d^{d-1}(\mu_A\otimes\id_V^{\otimes d-1})(\omega)\right)\\
=&-(\delta_{d,r}-\delta'_{d,r})(\omega)\\
=&\omega\otimes(\delta'_r-\delta_r),
\end{align*}
where the third equality holds by $W_{d+1}=0$, Proposition \ref{prop: relation betwenn different delta_r} and Theorem \ref{thm: properties of Koszul regular algebras}(c). Then the result follows by Proposition \ref{prop: relation betwenn different delta_r} again.
\end{proof}

\begin{definition}
Let $A=T(V)/(R)$ be a Koszul AS-regular algebra, $\overline{\sigma}$ a graded automorphism of $A$ and $\overline{\delta}$ a degree-one $\overline{\sigma}$-derivation of $A$. Let $(\delta_l,\delta_r)$ be the pair of elements in $V$ with respect to some sequence pair $\left(\{\delta_{i,r}\},\{\delta_{i,l}\}\right)$ for the $\overline{\delta}$. Then the element $\delta_r+\mu_A\overline{\sigma}^{-1}(\delta_l)$ is called the \emph{$\overline{\sigma}$-divergence} of $\overline{\delta}$, denoted by $\nabla_{\overline{\sigma}}\cdot\overline{\delta}$.
\end{definition}

\begin{remark}
If $A$ is a commutative graded polynomial algebra generated in degree 1 and $\overline{\sigma}$ is the identity map, $\nabla_{\id}\cdot\overline{\delta}$ is the usual divergence $\nabla\cdot\overline{\delta}$ of $\overline{\delta}$ (see Theorem \ref{thm: NA of Ore ext over Polynomial} or \cite[Thoerem 1.1(1)]{LM}). It motivates the name ``$\overline{\sigma}$-divergence'' of a $\overline{\sigma}$-derivation $\overline{\delta}$.
\end{remark}

\subsection{Ext-algebras}

In this subsection, we compute the Yoneda product of the Ext-algebra $E(B)$ partially.  We refer \cite{JZ} for the definition of homological determinant $\hdet(\overline{\sigma})$ of a graded automorphism $\overline{\sigma}$.

Following the minimal free resolution (\ref{resolution of koszul regular}) of $k_B$, one obtains that
\begin{align*}
&E^{1}(B)=\uHom_{B}(B(-1)\oplus W_1\otimes B,k)\cong k(1)\oplus W_1^*,\\
&E^{d}(B)=\uHom_{B}(W_{d-1}\otimes B(-1)\oplus W_{d}\otimes B,k)\cong W_{d-1}^*(1)\oplus W_d^*,\\
&E^{d+1}(B)=\uHom_{B}(W_{d}\otimes B(-1),k)\cong W_{d}^*(1).
\end{align*}
Then $E^{1}(B)$ has a basis $\xi, x_1^*,x_2^,\cdots,x^*_n$, where $\xi$ corresponds to the identity of $k$ (or the augmentation $\varepsilon_B:B(-1)\to k(-1)$), $E^{d}(B)$ has a basis $\eta_1^*,\eta_2^,\cdots,\eta^*_n,\omega^*$, and $E^{d+1}(B)$ has a basis $\widetilde{\omega}^*$ corresponding to  the element $\omega^*$ in $W_d^*$.
\begin{remark}
In order to make notations succinct, we  use the basis of $W_1, W_{d-1},W_d$ to represent the basis of $E^{1}(B),E^d(B),E^{d+1}(B)$. However, it should keep in mind that each element in such basis also has a corresponding homomorphism through the minimal free resolution. To be specific, each $x_i^*$ corresponds to $x_i^*\otimes \varepsilon_B:W_1\otimes B\to k(-1),$  $\eta_i^*$ corresponds to $\eta_i^*\otimes \varepsilon_B:W_{d-1}\otimes B(-1)\to k(-d)$,
$\omega^*$ corresponds to $\omega^*\otimes \varepsilon_B:W_{d}\otimes B\to k(-d)$ and $\widetilde{\omega}^*$ corresponds to $\omega^*\otimes \varepsilon_B:W_{d}\otimes B(-1)\to k(-d-1)$for $i=1,\cdots,n$. In the sequel, we use such correspondence freely.
\end{remark}

Define a graded algebra homomorphism
$$
\begin{array}{cclc}
p_z:  &  B &\to &k[z]\\
~  &  \sum_{i=1}^m a_iz^i &\mapsto & \sum_{i=1}^m \varepsilon_A(a_i)z^i.
\end{array}
$$
By \cite[Theorem 1]{SWZ}, $E(p_z)$ is a graded algebra homomorphism from $E(k[z])$ to $E(B)$. Notice that,
$$
0\to
k\{z\}\otimes k[z]\xlongrightarrow{h}
k[z]\xlongrightarrow{\varepsilon_{k[z]}}
k_{k[z]}\to
0,
$$
where $k\{z\}$ is the vector space spanned by $z$ and the right graded $k[z]$-module homomorphism $h$ mapping $z\otimes 1$ to $z$, is a minimal free resolution of the graded trivial module $k_{k[z]}$. So
$$
E^{1}(k[z])=\uHom_{k[z]}(k\{z\}\otimes k[z],k)=(k\{z\})^*,
$$
and $z^*$ (or $z^*\otimes \varepsilon_{k[z]}$) is a basis of $E^1(k[z])$.

\begin{lemma}
$E(p_z)(z^*)=\xi$.
\end{lemma}
\begin{proof}Clearly, we have the following commutative diagram
$$
\xymatrix{
W_1\otimes B(-1)\oplus W_{2}\otimes B \ar[r]\ar[d]^{0}&
%W_1\otimes B(-1)\oplus W_{2}\otimes B \ar[rr]^(0.55){\left({\scriptscriptstyle\begin{array}{ll}{\scriptstyle -\partial_{1} }& {\scriptstyle0} \\ {\scriptstyle\phi_1 }&{\scriptstyle \partial_{2}}\end{array}}\right)}
B(-1)\oplus W_1\otimes B\ar[rr]^(0.65){{\left(\begin{array}{ll}{\scriptstyle \phi_0}& {\scriptstyle \partial_1}\end{array}\right)}}
\ar[d]^{f}
&&
B\ar[r]^{\varepsilon_B}\ar[d]^{p_z}
&
k_B\ar[r]\ar@{=}[d]
&0
\\
0\ar[r]&
k\{z\}\otimes  k[z]\ar[rr]^{h}&&
k[z]\ar[r]^{\varepsilon_{k[z]}}
&
k_{k[z]}\ar[r]&
0,
}
$$
where $f(b,w_1\otimes b')=z\otimes p_z(b)$ for any $b,b'\in B$ and $w_1\in W_1$. By \cite[Theorem 1]{SWZ}, one obtains that
$$
E(p_z)(z^*)(b,w_1\otimes b')=(z^*\otimes\varepsilon_{k[z]})f(b,w_1\otimes b')=\varepsilon_{k[z]}p_z(b)=\varepsilon_{B}(b),
$$
as elements in $\uHom_{B}(B(-1)\oplus W_1\otimes B,k)$. Hence, $E(p_z)(z^*)=\xi$.
\end{proof}

Now it turns to compute the Yoneda product of $E^1(B)$ and $E^{d}(B)$. We fix a sequence pair $(\{\delta_{i,r}\},\{\delta_{i,l}\})$ for $\overline{\delta}$, and $(\delta_r,\delta_l)$ is the pair of elements in $V$ with respect to $(\{\delta_{i,r}\},\{\delta_{i,l}\})$.

\begin{lemma}\label{lemma: E^1(B) E^d(B)}
In $E(B)$, for any $i,j=1,\cdots,n$,
$$
\begin{array}{ll}
 \xi\ast\omega^*=(-1)^d\hdet(\overline{\sigma})\,\widetilde{\omega}^*,&\xi\ast\eta_j^*=0,\\
x^*_i\ast\omega^*=(-1)^dx_i^*(\delta_r)\cdot\widetilde{\omega}^*, &x_i^*\ast\eta_j^*=(-1)^{d+1}(\eta_j^*\otimes x^*_i)(\omega)\cdot \widetilde{\omega}^*.
\end{array}
$$
\end{lemma}
\begin{proof}
We claim that the following diagram is commutative.
$$
\xymatrix{
 W_d\otimes B(-1)\ar[r]^(0.38){\left({\scriptscriptstyle\begin{array}{l}{\scriptstyle -\partial_d} \\ {\scriptstyle\phi_d}\end{array}}\right)}\ar[d]^{\varphi_2}
 &
 W_{d-1}\otimes B(-1)\oplus W_d\otimes B\ar[d]^{\varphi_1} \ar[rd]^(0.6){\,(\eta_j^*\otimes\varepsilon_B,\  \omega^*\otimes\varepsilon_B)}\\
B(-d-1)\oplus W_1\otimes B(-d)\ar[r]^(0.63){(-1)^d\left(\begin{array}{ll}{\scriptstyle \phi_0}& {\scriptstyle \partial_1}\end{array}\right)}
&
B(-d)\ar[r]^{\varepsilon_B}&
k_B(-d),
}
$$
where $\varphi_1=(\eta_j^*\otimes\id_B,\ \omega^*\otimes\id_B)$, and
\begin{align*}
\varphi_2=(-1)^{d+1}\left(\begin{array}{c}
0\\
(\eta_j^*\otimes\id_V)\otimes\id_B
\end{array}
\right)+(-1)^d\left(\begin{array}{c}
\hdet(\overline{\sigma})(\omega^*\otimes \id_B)\\
(\omega^*\otimes\id_V)\delta_{d,r}\otimes \id_B
\end{array}
\right),
\end{align*}
and the first summand of $\varphi_2$ is induced by $\eta_j^*\otimes\varepsilon_B$ and the other one by $\omega^*\otimes\varepsilon_B$.

In fact, $\varepsilon_B\varphi_1=(\eta_j^*\otimes\varepsilon_B,\  \omega^*\otimes\varepsilon_B)$ is obvious, and
\begin{align*}
\varphi_1(-\partial_d\ \, \phi_d)^T&=-(\eta_j^*\otimes\id_B)\partial_d+(\omega^*\otimes\id_B)\phi_d\\
&=-\eta_j^*\otimes m_B+(\omega^*\otimes\id_B)(\sigma^{\otimes d}\otimes \lambda_z)+(\omega^*\otimes m_B)(\delta_{d,r}\otimes \id_B),\\
(-1)^d(\phi_0\ \partial_1)\varphi_2&=-\eta_j^*\otimes m_B+\hdet(\overline{\sigma})(\omega^*\otimes\lambda_z)+(\omega^*\otimes m_B)(\delta_{d,r}\otimes \id_B).
\end{align*}
By \cite[Theorem 1.2]{MS}, $\sigma^{\otimes d}=\hdet(\overline{\sigma})\cdot-:W_d\to W_d$. So the diagram is commutative. Then \begin{align*}
&\xi\ast(\eta_j^*,\  \omega^*)=(\varepsilon_B,0)\varphi_2=(-1)^d\hdet(\overline{\sigma})\,\omega^*\otimes\varepsilon_B,\\
&x^*_i\ast(\eta_j^*,\  \omega^*)=(0,x^*_i\otimes \varepsilon_B)\varphi_2=(-1)^{d+1}(\eta_j^*\otimes x^*_i)(\omega)\omega^*\otimes\varepsilon_B+(-1)^dx_i^*({\delta_r})\omega^*\otimes\varepsilon_B.
\end{align*}
The proof is completed.
\end{proof}

\begin{lemma}\label{lemma: E^d(B) E^1(B)}
In $E(B)$, for any $i,j=1,\cdots,n$,
$$
\begin{array}{ll}
 \omega^*\ast\xi= \widetilde{\omega}^*,&\eta_j^*\ast\xi=0,\\
\omega^*\ast x^*_i= -x_i^*(\delta_l)\widetilde{\omega}^*, &\eta_j^*\ast x_i^*=(-1)^d(x_i^*\sigma\otimes\eta_j^*)(\omega) \widetilde{\omega}^*.
\end{array}
$$
\end{lemma}

\begin{proof}Firstly, we check the following diagram is commutative.
$${\small
\xymatrix{
W_d\otimes B(-1)\ \ \ar[r]^(0.4){\scriptscriptstyle \left({\scriptscriptstyle\begin{array}{l}{\scriptstyle -\partial_d} \\ {\scriptstyle\phi_d}\end{array}}\right)}\ar[d]^{\psi_{d}}
& W_{d-1}\otimes B(-1)\oplus W_d\otimes B\ar[r]\ar[d]^{\psi_{d-1}}
&\cdots\ar[r]
&W_2\otimes B(-1)\oplus W_{3}\otimes B\ar[d]^{\psi_{2}}
\\
W_{d-1}\otimes B(-2)\oplus W_d\otimes B(-1)\ \ \ar[r]^{\scriptscriptstyle\left({\scriptscriptstyle\begin{array}{ll}{\scriptstyle \partial_{d-1} }& {\scriptstyle0} \\ {\scriptstyle-\phi_{d-1} }&{\scriptstyle -\partial_{d}}\end{array}}\right)}
&
\ \ \ W_{d-2}\otimes B(-2)\oplus W_{d-1}\otimes B(-1)\ar[r]
&\cdots\ar[r]
&W_1\otimes B(-2)\oplus W_{2}\otimes B(-1)
}
}$$
\begin{flushright}
${\small
\xymatrix{
\ar[rr]^(0.33){\left({\scriptscriptstyle\begin{array}{ll}{\scriptstyle -\partial_{2} }& {\scriptstyle0} \\ {\scriptstyle\phi_2 }&{\scriptstyle \partial_{3}}\end{array}}\right)}
&&W_1\otimes B(-1)\oplus W_{2}\otimes \ar[d]^{\psi_{1}} B\ar[rr]^(0.53){\left({\scriptscriptstyle\begin{array}{ll}{\scriptstyle -\partial_{1} }& {\scriptstyle0} \\ {\scriptstyle\phi_1 }&{\scriptstyle \partial_{2}}\end{array}}\right)}
&&B(-1)\oplus W_1\otimes B\ar[d]^{\psi_{0}}\ar[rd]^{(\varepsilon_B,\ x_i^*\otimes \varepsilon_B)}\\
\ar[rr]^(0.33){\left({\scriptscriptstyle\begin{array}{ll}{\scriptstyle \partial_{1} }& {\scriptstyle0} \\ {\scriptstyle-\phi_1 }&{\scriptstyle -\partial_{2}}\end{array}}\right)}
&&B(-2)\oplus W_1\otimes B(-1)\ar[rr]^(0.6){\left(\begin{array}{ll}{\scriptstyle -\phi_0}& {\scriptstyle -\partial_1}\end{array}\right)}
&&B(-1)\ar[r]^{\varepsilon_{B}}
&
k_B(-1),
}
}
\qquad$
\end{flushright}
where %$\psi_0=(\id_B,\ x_i^*\otimes\id_B)$,
\begin{align*}
%\psi_2&=
%\left(
%\begin{array}{cc}
%0 &  0\\
%\id_{W_{s-1}\otimes B(-1)}&0
%\end{array}
%\right)-\left(
%\begin{array}{cc}
%x_i^*\sigma\otimes\id_B &  0\\
%  \left[(x_i^*\otimes\id_V)\delta \right]\otimes\id_B &  - (x_i^*\otimes\id_V)\otimes\id_B
%\end{array}
%\right),\\
&\psi_0=(\id_B,\ x_i^*\otimes\id_B),\\
&\psi_s=
\left(
\begin{array}{cc}
0 &  0\\
\id_{W_{s}\otimes B(-1)}&0
\end{array}
\right)+
(-1)^{s}
\left(
\begin{array}{cc}
(x_i^*\sigma\otimes\id_V^{\otimes s-1})\otimes\id_B &  0\\
  (-1)^{s-1}\left[(x_i^*\otimes\id_V^{\otimes s})\delta_{s,l} \right]\otimes\id_B &  (x_i^*\otimes\id_V^{\otimes s})\otimes\id_B
\end{array}
\right),\ \  \ 1\leq s\leq d-1,\\
&\psi_{d}=
\left(
\begin{array}{c}
0 \\
\id_{W_{d}\otimes B(-1)}
\end{array}
\right)+
(-1)^d
\left(
\begin{array}{c}
(x_i^*\sigma\otimes\id_V^{\otimes d-1})\otimes\id_B \\
(-1)^{d-1}\left[(x_i^*\otimes\id_V^{\otimes d})\delta_{d,l} \right]\otimes\id_B
\end{array}
\right),
\end{align*}
and the first summand of those maps is induced by $\xi$ and the other one by $x_i^*\otimes \varepsilon_B$.

Write $\partial_i^B$ for the $i$-th differential in the minimal resolution of $k_B$ for $i\geq 1$. Obviously, $\varepsilon_B\psi_0=(\varepsilon_B,\ x_i^*\otimes\varepsilon_B)$. Also, $\psi_0\partial^B_2=-\partial^B_0\psi_1$, since $\delta=\delta_{1,r}=\delta_{1,l}$ and
\begin{align*}
\psi_0\partial^B_2&=\left(-\partial_1+(x_i^*\otimes\id_B)\phi_1,\ (x_i^*\otimes\id_B)\partial_2      \right)\\
&=\left(-\partial_1+x_i^*\sigma\otimes\lambda_z+(x_i^*\otimes m_B)(\delta_{1,r}\otimes \id_B),\ x_i^*\otimes m_B    \right),\\
-\partial^B_0\psi_1&=\left(
-\partial_1+\phi_0(x^*_i\sigma\otimes\id_B)+\partial_1\left[(x_1^*\otimes\id_V)\delta_{1,l}\otimes\id_B\right],\ \partial_1((x_i^*\otimes\id_V)\otimes \id_B)
         \right)\\
&=\left(
-\partial_1+x^*_i\sigma\otimes\lambda_z+(x_1^*\otimes m_B)(\delta_{1,l}\otimes\id_B),\ x_i^*\otimes m_B\right).
\end{align*}
For $1\leq s\leq d-2$,
\begin{align*}
\psi_{s}\partial_{s+2}^B&=
\left(
\begin{array}{cc}
0 &  0\\
-\partial_{s+1}&0
\end{array}
\right)
+
(-1)^s
\left(
\begin{array}{cc}
-(x_i^*\sigma\otimes\id_V^{\otimes s-2})\otimes m_B &  0\\
  \zeta_s&  (x_i^*\otimes\id_V^{\otimes s-1})\otimes m_B
\end{array}
\right),\\
%&\left(
%\begin{array}{cc}
%-(x_i^*\sigma\otimes\id_V^{\otimes s-2})\otimes m_B &  0\\
%  \left[(x_i^*\otimes\id_V^{\otimes s-1})\delta_2 \right]\otimes m_B+(x^*_i\sigma\otimes \sigma^{\otimes s-1})\otimes \lambda_z+(x^*_i\otimes\id_V^{\otimes s-1}\otimes m_B)(\delta_1\otimes\id_B) &   (x_i^*\otimes\id_V^{\otimes s-1})\otimes m_B
%\end{array}
%\right)\\
-\partial_{s+1}^B\psi_{s+1}&=\left(
\begin{array}{cc}
0 &  0\\
-\partial_{s+1}&0
\end{array}
\right)+
(-1)^{s}\left(
\begin{array}{cc}
-(x_i^*\sigma\otimes\id_V^{\otimes s-2})\otimes m_B &  0\\
 \zeta'_s&   (x_i^*\otimes\id_V^{\otimes s-1})\otimes m_B
\end{array}
\right),
\end{align*}
where
\begin{align*}
\zeta_s&=(-1)^s\left(\left((x_i^*\otimes\id_V^{\otimes s})\delta_{s,l} \right)\otimes\id_B\right)\partial_{s+1} +\left((x_i^*\otimes\id_V^{\otimes s})\otimes\id_B \right)\phi_{s+1}\\
&=(-1)^s\left((x_i^*\otimes\id_V^{\otimes s})\delta_{s,l} \right)\otimes  m_B+\left(x_i^*\sigma \otimes\sigma^{\otimes s}\right)\otimes\lambda_z+\left(x_i^*\otimes\id_V^{\otimes s}\otimes m_B\right)(\delta_{s+1,r}\otimes\id_B)\\
&=\left(x_i^*\sigma \otimes\sigma^{\otimes s}\right)\otimes\lambda_z+(-1)^s\left(x_i^*\otimes\id_V^{\otimes s}\otimes m_B\right)\left((\delta_{s,l}\otimes\id_V)\otimes\id_B \right)+\left(x_i^*\otimes\id_V^{\otimes s}\otimes m_B\right)\left(\delta_{s+1,r}\otimes\id_B\right),\\
\zeta'_s&=\phi_{s} \left((x_i^*\sigma\otimes\id_V^{\otimes s})\otimes\id_B\right)   +(-1)^s\partial_{s+1}    \left(\left((x_i^*\otimes\id_V^{\otimes s+1})\delta_{s+1,l} \right)\otimes\id_B\right)\\
&=\left(x_i^*\sigma\otimes\sigma^{\otimes s}\right)\otimes\lambda_z+\left(\id_V^{\otimes s}\otimes m_B\right)\left((x_i^*\sigma\otimes\delta_{s,r})\otimes\id_B\right)+(-1)^s\left(x_i^*\otimes\id_V^{\otimes s}\otimes m_B\right)\left(\delta_{s+1,l} \otimes\id_B\right)\\
&=\left(x_i^*\sigma\otimes\sigma^{\otimes s}\right)\otimes\lambda_z+\left(x_i^*\otimes\id_V^{\otimes s}\otimes m_B\right)\left((\sigma\otimes\delta_{s,r})\otimes\id_B\right)+(-1)^s\left(x_i^*\otimes\id_V^{\otimes s}\otimes m_B\right)\left(\delta_{s+1,l} \otimes\id_B\right).
\end{align*}
By Lemma \ref{lemma: construction of delta_l}, $\zeta_s=\zeta_s'$ and $\psi_s\partial^B_{s+2}=-\partial_{s+1}^B\psi_{s+1}$. Similarly, one obtains $\psi_{d-1}\partial^B_{d+1}=-\partial_{d}^B\psi_{d}$.

Then, for any $i,j=1,\cdots,n$, we have in $E(B)$,
 \begin{align*}
&\eta^*_j\ast(\xi,x_i^*)=(\eta_j^*\otimes\varepsilon_B)\ast(\varepsilon_B,x_i^*\otimes \varepsilon_B)=(\eta_j^*\otimes\varepsilon_B,0)\psi_d=(-1)^d(x_i^*\sigma\otimes\eta_j^*)(\omega)\omega^*\otimes \varepsilon_B,\\
&\omega^*\ast(\xi,x_i^*)=(\omega^*\otimes\varepsilon_B)\ast(\varepsilon_B,x_i^*\otimes \varepsilon_B)=(0,\omega^*\otimes\varepsilon_B)\psi_d=\omega^*\otimes \varepsilon_B-x_i^*(\delta_l)\omega^*\otimes \varepsilon_B .
\end{align*}
The result follows.
\end{proof}

\subsection{Nakayama automorphisms}
This subsection devotes to proving the main result of this paper. For the completeness, we give a whole computation of the Nakayama automorphism of a graded Ore extension of a Koszul AS-regular algebra, which includes \cite[Proposition 3.15]{ZVZ} partially. For the automorphism $\sigma\in GL(V)$ and the Nakayama automorphism $\mu_A$ of $A$, there exist two invertible $n\times n$ matrixes $M=(m_{ij}),P=(p_{ij})$ over $k$, such that
$$\sigma
\left(
\begin{array}{c}
x_1\\
x_2\\
\vdots\\
x_n
\end{array}
\right)=M\left(
\begin{array}{c}
x_1\\
x_2\\
\vdots\\
x_n
\end{array}
\right),\qquad
\mu_A
\left(
\begin{array}{c}
x_1\\
x_2\\
\vdots\\
x_n
\end{array}
\right)=P\left(
\begin{array}{c}
x_1\\
x_2\\
\vdots\\
x_n
\end{array}
\right).
$$

\begin{lemma}\label{lemma: Yoneda product in E(A)}
For any $j=1,\cdots,n$,
%$$
%(\eta_j^*\otimes x^*_i)(\omega)=(-1)^{d-1}\sum_{s=1}^{n}p_{si}(x_i^*\otimes\eta^*_j)(\omega).
%$$
$$
\left(
\begin{array}{c}
\eta_j^*\otimes x^*_1(\omega)\\
\eta_j^*\otimes x^*_2(\omega)\\
\vdots\\
\eta_j^*\otimes x^*_n(\omega)
\end{array}
\right)
=(-1)^{d-1}P^T\left(
\begin{array}{c}
 x^*_1\otimes\eta^*_j(\omega)\\
x^*_2\otimes\eta^*_j(\omega)\\
\vdots\\
x^*_n\otimes\eta^*_j(\omega)
\end{array}
\right).
$$
\end{lemma}
\begin{proof}Since $A$ is a Koszul AS-regular algebra of dimension $d$, the minimal free  resolution (\ref{resolution of k_A}) of trivial module $k_A$ becomes
$$
0\xlongrightarrow{}W_d\otimes A\xlongrightarrow{\partial^A_d}W_{d-1}\otimes A\xlongrightarrow{\partial^A_{d-1}}\cdots\xlongrightarrow{\partial^A_{3}}W_2\otimes A\xlongrightarrow{\partial^A_{2}}W_1\otimes A\xlongrightarrow{\partial^A_{1}}A\xlongrightarrow{\varepsilon_A}k_A\to0.
$$
So $E^1(A)=W_1^*$, $E^{d-1}(A)=W_{d-1}^*$, $E^{d}(A)=W_{d}^*$.  By a similar argument in the proof of Lemma \ref{lemma: E^1(B) E^d(B)} and Lemma \ref{lemma: E^d(B) E^1(B)}, one obtains that the Yoneda products in $E(A)$ of $\{x_i^*\}_{i=1}^{n}$ and $\{\eta_{j}^*\}_{j=1}^n$, which are basis of $E^1(A)$ and $E^{d-1}(A)$ respectively, are as follows:
$$
x_i^*\ast \eta_j^*= (-1)^{d-1}(\eta_j^*\otimes x_i^*)(\omega)\omega^*,\qquad \eta_j^*\ast x_i^*=(-1)^{d-1}(x_i^*\otimes\eta^*_j)(\omega)\omega^*.
$$
By  Theorem  \ref{thm: properties of Koszul regular algebras}(a,b), $\mu_{E(A)}(x_i^*)=\sum_{s=1}^{n}p_{si}x_s^*$, and $E(A)$ is graded Frobenius with the bilinear form as follows
\begin{align*}
&\langle x_i^*, \eta_j^*\rangle=(x_i^*\ast \eta_j^*)(\omega)=(-1)^{d-1}(\eta_j^*\otimes x^*_i)(\omega),\\
&\langle \eta_j^*, \mu_{E(A)}(x_i^*)\rangle=\sum_{s=1}^{n}p_{si}\langle \eta_j^*, x_s^*\rangle=\sum_{s=1}^{n}p_{si}(\eta_j^*\ast x_s^*)(\omega)=(-1)^{d-1}\sum_{s=1}^{n}p_{si}(x_s^*\otimes\eta^*_j)(\omega),
\end{align*}
%$$
%(\eta_j^*\otimes x^*_i)(\omega)=(-1)^{d-1}\langle x_i^*, \eta_j^*\rangle=\langle \eta_j^*, \mu_{E(A)}(x_i^*)\rangle=\sum_{s=1}^{n}p_{si}\langle \eta_j^*, x_s^*\rangle=(-1)^{d-1}\sum_{s=1}^{n}p_{si}(x_i^*\otimes\eta^*_j)(\omega),
%$$
for any $i,j=1,\cdots,n$. The result follows.
\end{proof}

Now we prove the main result of this paper.

\begin{theorem}\label{thm: nakayama automorphism of ore extension} Let $B=A[z;\overline{\sigma},\overline{\delta}]$ be a graded Ore extension of a Koszul AS-regular algebra $A$, where $\overline{\sigma}$ is a graded automorphism of $A$ and $\overline{\delta}$ is a degree-one $\overline{\sigma}$-derivation. Then the Nakayama automorphism $\mu_B$ of $B$ satisfies
\begin{align*}
{\mu_{B}}_{\mid A}=\overline{\sigma}^{~-1}\mu_A,\qquad
%&\mu_B(z)=\hdet\sigma\ z+c_1x_1+c_2+\cdots+c_nx_n,\\
\mu_B(z)=\hdet(\overline{\sigma})\, z+\nabla_{\overline{\sigma}}\cdot\overline{\delta},
\end{align*}
where $\mu_A$ is the Nakayama automorphism of $A$ and $\nabla_{\overline{\sigma}}\cdot\overline{\delta}$ is the $\overline{\sigma}$-divergence of $\overline{\delta}$.
%where
%$$
%\left(
%\begin{array}{c}
%c_1\\
%c_2\\
%\vdots\\
%c_n
%\end{array}
%\right)
%=
%\left(
%\begin{array}{c}
% x_1^*(\delta_r)\\
% x_2^*(\delta_r)\\
%\vdots\\
% x_n^*(\delta_r)
%\end{array}
%\right)+(M^{-1}P)^T
%\left(
%\begin{array}{c}
% x_1^*(\delta_l)\\
% x_2^*(\delta_l)\\
%\vdots\\
% x_n^*(\delta_l)
%\end{array}
%\right).
%$$
\end{theorem}
\begin{proof}
Let $(\delta_r,\delta_l)$ be the pair of elements in $V$ with respect to some sequence pair $(\{\delta_{i,r}\},\{\delta_{i,l}\})$ for $\overline{\delta}$.
%$\delta:V\to V\otimes V$ be a linear map which induces the map $\overline{\delta}$, and

By Lemma \ref{lemma: E^1(B) E^d(B)} and Lemma \ref{lemma: E^d(B) E^1(B)},
\begin{align*}
&\langle \xi, \omega^*\rangle=(-1)^d\hdet(\overline{\sigma})=(-1)^d\langle \omega^*, \hdet(\overline{\sigma})\, \xi\rangle,\\
&\langle \xi, \eta_j^*\rangle=0=(-1)^d\langle \eta_j^*, \hdet(\overline{\sigma})\, \xi\rangle,
\end{align*}
for any $j=1,\cdots,n$. Hence,
$$\mu_{E(B)}(\xi)=\hdet(\overline{\sigma})\,\xi.$$

Write $\omega=x_1\otimes \upsilon_1+x_2\otimes\upsilon_2+\cdots+x_n\otimes \upsilon_n$, where $\upsilon_1,\upsilon_2,\cdots,\upsilon_n\in W_{d-1}$. Then $(x_i^*\otimes \eta_j^*)(\omega)=\eta_j^*(\upsilon_i)$, and
$$
(x_i^*\sigma\otimes \eta_j^*)(\omega)=(x_i^*\otimes\eta^*_{j})\left(\sum_{s,t=1}^n m_{st}(x_t\otimes \upsilon_s)\right)=\sum_{s=1}^n m_{si}\eta_{j}^*(\upsilon_s).
$$
By Lemma \ref{lemma: Yoneda product in E(A)}, one obtains
$$
\left(
\begin{array}{c}
(x_1^*\sigma\otimes \eta_j^*)(\omega)\\
(x_2^*\sigma\otimes \eta_j^*)(\omega)\\
\vdots\\
(x_n^*\sigma\otimes \eta_j^*)(\omega)
\end{array}
\right)
=M^T\left(
\begin{array}{c}
\eta_j^*(\upsilon_1)\\
\eta_j^*(\upsilon_2)\\
\vdots\\
\eta_j^*(\upsilon_n)
\end{array}
\right)=M^T\left(
\begin{array}{c}
(x_1^*\otimes \eta_j^*)(\omega)\\
(x_2^*\otimes \eta_j^*)(\omega)\\
\vdots\\
(x_n^*\otimes \eta_j^*)(\omega)
\end{array}
\right)=(-1)^{d-1}(P^{-1}M)^T
\left(
\begin{array}{c}
(\eta_j^*\otimes x_1^*)(\omega)\\
(\eta_j^*\otimes x_2^*)(\omega)\\
\vdots\\
(\eta_j^*\otimes x_n^*)(\omega)
\end{array}
\right).
$$
By Lemma \ref{lemma: E^1(B) E^d(B)} and Lemma \ref{lemma: E^d(B) E^1(B)},
$$
\left(
\begin{array}{c}
\langle x_1^*,\eta_j^*\rangle\\
\langle x_2^*,\eta_j^*\rangle\\
\vdots\\
\langle x_n^*,\eta_j^*\rangle
\end{array}
\right)=(-1)^{d+1}\left(
\begin{array}{c}
(\eta_j^*\otimes x_1^*)(\omega)\\
(\eta_j^*\otimes x_2^*)(\omega)\\
\vdots\\
(\eta_j^*\otimes x_n^*)(\omega)
\end{array}
\right)=(M^{-1}P)^T\left(
\begin{array}{c}
(x_1^*\sigma\otimes \eta_j^*)(\omega)\\
(x_2^*\sigma\otimes \eta_j^*)(\omega)\\
\vdots\\
(x_n^*\sigma\otimes \eta_j^*)(\omega)
\end{array}
\right)=(-1)^d(M^{-1}P)^T
\left(
\begin{array}{c}
\langle \eta_j^*,x_1^*\rangle\\
\langle \eta_j^*,x_2^*\rangle\\
\vdots\\
\langle \eta_j^*,x_n^*\rangle
\end{array}
\right).
$$

%Write
%$$
%\left(
%\begin{array}{c}
%c_1\\
%c_2\\
%\vdots\\
%c_n
%\end{array}
%\right)
%=
%\left(
%\begin{array}{c}
% x_1^*(\delta_r)\\
% x_2^*(\delta_r)\\
%\vdots\\
% x_n^*(\delta_r)
%\end{array}
%\right)-(-1)^d(M^{-1}P)^T
%\left(
%\begin{array}{c}
% x_1^*(\delta_l)\\
% x_2^*(\delta_l)\\
%\vdots\\
% x_n^*(\delta_l)
%\end{array}
%\right).
%$$
Write $c_i=x_i^*(\nabla_{\overline{\sigma}}\cdot\overline{\delta})$ for any $i=1,\cdots,n$, and it is easy to see that
$$
\left(
\begin{array}{c}
c_1\\
c_2\\
\vdots\\
c_n
\end{array}
\right)
=
\left(
\begin{array}{c}
 x_1^*(\delta_r)\\
 x_2^*(\delta_r)\\
\vdots\\
 x_n^*(\delta_r)
\end{array}
\right)+(M^{-1}P)^T
\left(
\begin{array}{c}
 x_1^*(\delta_l)\\
 x_2^*(\delta_l)\\
\vdots\\
 x_n^*(\delta_l)
\end{array}
\right).
$$
Then
\begin{align*}
\left(
\begin{array}{c}
\langle x_1^*,\eta_j^*\rangle\\
\langle x_2^*,\eta_j^*\rangle\\
\vdots\\
\langle x_n^*,\eta_j^*\rangle
\end{array}
\right)&=(-1)^d\left((M^{-1}P)^T
\left(
\begin{array}{c}
\langle \eta_j^*,x_1^*\rangle\\
\langle \eta_j^*,x_2^*\rangle\\
\vdots\\
\langle \eta_j^*,x_n^*\rangle
\end{array}
\right)+
\left(
\begin{array}{c}
\langle \eta_j^*,c_1\xi\rangle\\
\langle \eta_j^*,c_2\xi\rangle\\
\vdots\\
\langle \eta_j^*,c_n\xi\rangle
\end{array}
\right)
\right),
\\
\left(
\begin{array}{c}
\langle x_1^*,\omega^*\rangle\\
\langle x_2^*,\omega^*\rangle\\
\vdots\\
\langle x_n^*,\omega^*\rangle
\end{array}
\right)
&=(-1)^d\left(
\begin{array}{c}
 x_1^*(\delta_r)\\
 x_2^*(\delta_r)\\
\vdots\\
 x_n^*(\delta_r)
\end{array}
\right)=
(-1)^d\left((M^{-1}P)^T
\left(
\begin{array}{c}
\langle \omega^*,x_1^*\rangle\\
\langle \omega^*,x_2^*\rangle\\
\vdots\\
\langle \omega^*,x_n^*\rangle
\end{array}
\right)+
\left(
\begin{array}{c}
\langle \omega^*,c_1\xi\rangle\\
\langle \omega^*,c_2\xi\rangle\\
\vdots\\
\langle \omega^*,c_n\xi\rangle
\end{array}
\right)
\right).
\end{align*}
Hence,
$$
\mu_{E(B)}
\left(
\begin{array}{c}
x_1^*\\
x_2^*\\
\vdots\\
x_n
\end{array}
\right)
=(M^{-1}P)^T\left(
\begin{array}{c}
x_1^*\\
x_2^*\\
\vdots\\
x_n
\end{array}
\right)+\left(
\begin{array}{c}
c_1\xi\\
c_2\xi\\
\vdots\\
c_n\xi
\end{array}
\right).
$$
By Theorem \ref{thm: properties of Koszul regular algebras}(b), we have
\begin{align*}
&{\mu_{B}}_{\mid A}=\overline{\sigma}^{-1}\mu_A,\\
&\mu_B(z)=\hdet(\overline{\sigma})\, z+c_1x_1+c_2+\cdots+c_nx_n=\hdet(\overline{\sigma})\, z+\nabla_{\overline{\sigma}}\cdot\overline{\delta}\qedhere
\end{align*}
\end{proof}

\begin{corollary}
Let $A$ be a Koszul AS-regular algebra with the Nakayama automorphism $\mu_A$. Then a graded Ore extension $B=A[z;\overline{\sigma},\overline{\delta}]$ is Calabi-Yau if and only if $\overline{\sigma}=\mu_A$ and $\nabla_{\overline{\sigma}}\cdot\overline{\delta}=0$.
%if and only if $\overline{\sigma}=\mu_A$ and $\delta_r=-\delta_l$ for ....
\end{corollary}
%
%\begin{example}
%Let $A=k[x]$ be the only Koszul AS-regular of dimension $1$.
%
%\end{example}

\subsection{Twisted superpotentials}As shown in Theorem \ref{thm: properties of Koszul regular algebras}(c,d), a  Koszul AS-regular algebra is always associated with a twisted superpotential such that it is a derivation quotient algebra defined by such a twisted superpotential. Since graded Ore extensions of Koszul AS-regular algebras are also Koszul AS-regular algebras, it is worth to understand twisted superpotentials for graded Ore extensions. In \cite{HVZ,HVZ1}, the authors studied this problem in two special cases. In the following, we give a general solution to this problem.

\begin{theorem}\label{thm: twisted superpotential for B} Let $A=T(V)/(R)$ be a Koszul AS-regular algebra of dimension $d$ and $\omega$ a basis of $W_d$. Suppose $B=A[z;\overline{\sigma},\overline{\delta}]$ is a graded Ore extension of $A$, where $\overline{\sigma}$ is a graded automorphism of $A$ and $\overline{\delta}$ is a degree-one $\overline{\sigma}$-derivation of $A$. Let $(\{\delta_{i,r}\},\{\delta_{i,l}\})$ be a sequence pair for $\overline{\delta}$, then
\begin{align}
\hat{\omega}&=\sum_{i=0}^d (-1)^i\tau_{d+1}^i(\id\otimes \sigma^{\otimes i}\otimes\id_V^{\otimes d-i})(z\otimes\omega)+\sum_{i=1}^d (-1)^i(\delta_{i,r}\otimes\id_V^{\otimes d-i})(\omega)\label{twsited superpotential1}\\
&=\sum_{i=0}^d (-1)^i\tau_{d+1}^i(\id\otimes \sigma^{\otimes i}\otimes\id_V^{\otimes d-i})(z\otimes\omega)+(-1)^{d+1}\sum_{i=1}^d (-1)^{i}(\sigma^{\otimes d-i}\otimes \delta_{i,l})(\omega).\label{twsited superpotential2}
\end{align}
is a ${\mu_B}_{\mid V}$-twisted superpotential, where $\mu_B$ is the Nakayama automorphism of $B$ and $\sigma=\overline{\sigma}_{\mid V}$. Moreover,
$$B\cong \mathcal{A}(\hat{\omega},d-1).$$
\end{theorem}
\begin{proof}
By Proposition \ref{prop: relation between delta_r and delta_l}(a), one obtains that (\ref{twsited superpotential1}) and (\ref{twsited superpotential2}) are equal. By Theorem \ref{thm: properties of Koszul regular algebras}(c), $$\tau_{d}^{d-1}(\mu_A\otimes\id_V^{\otimes d})(\omega)=(-1)^{d-1}\omega.$$
It remains to show  $\tau_{d+1}^{d}(\mu_B\otimes\id_V^{\otimes d})(\hat{\omega})=(-1)^d\hat{\omega}$. By Theorem \ref{thm: nakayama automorphism of ore extension}, we have
\begin{align*}
\tau_{d+1}^{d}(\mu_B\otimes\id_V^{\otimes d})(\hat{\omega})&=\tau_{d+1}^{d}(\mu_B\otimes\id_V^{\otimes d})\left(\sum_{i=0}^d (-1)^i\tau_{d+1}^i(\id\otimes \sigma^{\otimes i}\otimes\id_V^{\otimes d-i})(z\otimes\omega)-\sum_{i=1}^d (-1)^{d-i}(\sigma^{\otimes d-i}\otimes \delta_{i,l})(\omega)\right)\\
&=\omega\otimes\mu_B(z)+\sum_{i=1}^d (-1)^{i}\tau_{d+1}^{d}\tau_{d+1}^i(\id\otimes\mu_A\otimes \sigma^{\otimes i-1}\otimes\id_V^{\otimes d-i})(z\otimes\omega)\\
&\quad -\sum_{i=1}^{d-1} (-1)^{d-i}\tau_{d+1}^{d}(\mu_A\otimes\sigma^{\otimes d-i-1}\otimes \delta_{i,l})(\omega)-\tau_{d+1}^{d}(\mu_B\otimes\id_V^{\otimes d})\delta_{d,l}(\omega)\\
&=\omega\otimes\hdet(\overline{\sigma})\, z+\omega\otimes\nabla_{\overline{\sigma}}\cdot\overline{\delta}\\
&\quad+\sum_{i=1}^d (-1)^{i}\tau_{d+1}^{i-1}(\id\otimes\sigma^{\otimes i-1}\otimes\id_V^{\otimes d-i+1})\left(z\otimes\left(\tau_d^{d-1}(\mu_A\otimes\id_V^{\otimes d-1})(\omega)\right)\right)\\
&\quad -\sum_{i=1}^{d-1} (-1)^{d-i}(\sigma^{\otimes d-i-1}\otimes \delta_{i,l}\otimes\id_V)\tau_{d}^{d-1}(\mu_A\otimes\id_V^{\otimes d-1})(\omega)-\omega\otimes\mu_A\sigma^{-1}(\delta_l)\\
&=\omega\otimes\hdet(\overline{\sigma})\, z+\omega\otimes\delta_r+\sum_{i=1}^d (-1)^{d+i-1}\tau_{d+1}^{i-1}(\id\otimes\sigma^{\otimes i-1}\otimes\id_V^{\otimes d-i+1})(z\otimes\omega)\\
&\quad-\sum_{i=1}^{d-1} (-1)^{i+1}(\sigma^{\otimes d-i-1}\otimes \delta_{i,l}\otimes\id_V)(\omega)\\
&=\tau_{d+1}^{d}(\id\otimes \sigma^{\otimes d})(z\otimes \omega)+\delta_{d,r}(\omega)+\sum_{i=0}^{d-1} (-1)^{d+i}\tau_{d+1}^{i}(\id\otimes\sigma^{\otimes i}\otimes\id_V^{\otimes d-i})(z\otimes\omega)\\
&\quad+(-1)^{d}\sum_{i=1}^{d-1}(-1)^i(\delta_{i,r}\otimes\id_V^{\otimes d-i})(\omega)\\
&=(-1)^d\left(\sum_{i=0}^{d} (-1)^{i}\tau_{d+1}^{i}(\id\otimes\sigma^{\otimes i}\otimes\id_V^{\otimes d-i})(z\otimes\omega)+\sum_{i=1}^{d}(-1)^i(\delta_{i,r}\otimes\id_V^{\otimes d-i})      \right)(\omega)\\
&=(-1)^d\hat{\omega},
\end{align*}
where the third equation holds by \cite[Theorem 1.2]{MS} and Proposition \ref{prop: relation between delta_r and delta_l}(b).

Let $\delta:V\to V\otimes V$ be a linear map such that the map $\overline{\delta}$ and  the sequence pair $(\{\delta_{i,r}\},\{\delta_{i,l}\})$ are   induced by $\delta$, and $\{x_1,\cdots,x_n\}$ a basis of $V$. Write
$$
\hat{V}=V\oplus k\{z\},\qquad  \hat{R}=R\oplus k\{z\otimes x_i-\sigma(x_i)\otimes z-\delta(x_i)\mid i=1,\cdots,n\}.
$$
Then $B\cong T(\hat{V})/(\hat{R})$ is a Koszul regular algebra of dimension $d+1$. Write
$$
\hat{W}_1=\hat{V},\qquad  \hat{W}_i=\bigcap_{0\leq s\leq i-2} \hat{V}^{\otimes s}\otimes \hat{R}\otimes \hat{V}^{\otimes i-s-2},\quad \forall i\geq 2.
$$
Clearly, $W_i\subseteq \hat{W_i}$ for any $i\geq 1$ and $\dim \hat{W}_{d+1}=1$.

It is easy to know that $\hat{\omega}\neq 0$. Since $B$ is Koszul AS-regular, it suffices to prove $\hat{\omega}\in \hat{W}_{d+1}$ by Theorem \ref{thm: properties of Koszul regular algebras}(d). Write $\omega=\sum v_1\otimes v_2\otimes \cdots\otimes v_d$. By (\ref{twsited superpotential1}), one obtains that
\begin{align*}
\hat{\omega}=&\sum \left(z\otimes v_1-\sigma(v_1)\otimes z-\delta(v_1)\right)\otimes v_2\otimes \cdots\otimes v_d\\
&+
\sum_{i=2}^d (-1)^i(\sigma(v_1)\otimes \sigma(v_2)\otimes \cdots\otimes  \sigma(v_i)\otimes z\otimes v_{i+1}\otimes\cdots\otimes v_d)
+\sum_{i=2}^d (-1)^i(\delta_{i,r}\otimes\id_V^{\otimes d-i})(\omega)\in \hat{R}\otimes \hat{V}^{\otimes d-1}.
\end{align*}
Since $\underbrace{\left((\tau_{d+1}^d)^{-1}\circ\cdots(\tau_{d+1}^d)^{-1}\right)}_i(\hat{R}\otimes \hat{V}^{\otimes d-1})\subseteq \hat{V}^{\otimes i}\otimes \hat{R}\otimes \hat{V}^{\otimes d-1-i}$ and $\hat{\omega}$ is a ${\mu_B}_{\mid V}$-twisted superpotential, we have
$$
\hat{\omega}=(-1)^{di}\underbrace{\left((\mu_B^{-1}\otimes \id_V^{\otimes d})(\tau_{d}^{d+1})^{-1}\cdots (\mu_B^{-1}\otimes \id_V^{\otimes d})(\tau_{d}^{d+1})^{-1}\right)}_i(\hat{\omega})\in \hat{V}^{\otimes i}\otimes \hat{R}\otimes \hat{V}^{\otimes d-1-i},
$$
for any $1\leq i\leq d-1$. It implies that $\hat{\omega}\in\hat{W}_{d+1}$.
%\begin{align*}
%&\sum_{i=1}^d (-1)^i(\delta_{i,r}\otimes\id_V^{\otimes d-i})+\sum_{i=1}^d (-1)^{d-i}(\sigma^{\otimes d-i}\otimes \delta_{i,l})\\
%=&(-1)^d(\delta_{d,r}+(-1)^d\delta_{d,l})+\sum_{i=1}^{d-1}\left((-1)^i\delta_{i,r}\otimes\id_V^{\otimes d-i}+(-1)^{d-i} \sigma^{\otimes d-i}\otimes \delta_{i,l} \right)\\
%=&(-1)^d\left(\sigma\otimes(\delta_{d-1,r}+(-1)^{d-1} \delta_{d-1,l})-(\delta_{d-1,r}+(-1)^{d-1}\delta_{d-1,l}
%)\otimes\id_V\right)\\
%&+\sum_{i=1}^{d-2}\left((-1)^i\delta_{i,r}\otimes\id_V^{\otimes d-i}+(-1)^{d-i} \sigma^{\otimes d-i}\otimes \delta_{i,l} \right)
%\end{align*}
\end{proof}

\begin{remark}
\begin{enumerate}
\item By Proposition \ref{prop: relation betwenn different delta_r}, the twisted superpotential $\hat{\omega}$ constructed in the last theorem is independent on the choices of sequence pairs for $\overline{\delta}$.
\item The results in \cite[Theorem 4.4]{HVZ} and \cite[Theorem 0.1(ii)]{HVZ1} are both the special case of $\delta_{i,r}=\delta_{i,l}=0$ for $i\geq 2$.
\end{enumerate}
\end{remark}

%It is a general version of the construction of twisted superpotentials for graded Ore extensions.

\section{Applications}
In this section, we apply our main result to two examples.

\subsection{Graded polynomial algebras} In this part, we assume $A=k[x_1,x_2,\cdots,x_n]$ is a graded polynomial algebra generated in degree $1$. The Nakayama automorphism of a graded Ore extension of $A$ is just a graded version of \cite[Theorem 1.1]{LM}. We use our method to  prove the differential case as an example.

\begin{theorem}\label{thm: NA of Ore ext over Polynomial}\cite[Theorem 1.1(1)]{LM} Let $A=k[x_1,x_2,\cdots,x_n]$ be a graded polynomial algebra generated in degree $1$. Then the Nakayama automorphism $\mu_B$ of the graded Ore extension $B=A[z;\overline{\delta}]$ is
\begin{align*}
{\mu_B}_{|A}=\id_A\qquad
\mu_B(z)=z+\nabla\cdot {\overline{\delta}};
\end{align*}
where $\nabla\cdot {\bar{\delta}}$ is the divergence of $\delta$, that is $\nabla\cdot {\bar{\delta}}=\sum_{i=1}^n \partial\, \overline{\delta}(x_i)/\partial x_i$.
\end{theorem}

To prove this theorem, we need some preparation. Let $V$ be the vector space spanned by $\{x_1,x_2,\cdots,x_n\}$. Firstly, we determine the vector spaces $\{W_i\mid i\geq2\}$ for $A$. Write $
r_{i_1i_2}=x_{i_1}\otimes x_{i_2}-x_{i_2}\otimes x_{i_1},
$
for any (not necessarily distinguished) $i_1,i_2\in\{1,2,\cdots,n\}$. For any integer $m\geq3$, we write inductively,
%$$
%r_{i_1i_2\cdots i_m}=\sum_{\substack{\sigma\in S_m\\ \sigma(1)<\sigma(2)<\cdots<\sigma(m-1)}}(-1)^{
%\mathrm{sgn}\sigma}r_{i_{\sigma(1)}i_{\sigma(2)}\cdots i_{\sigma(m-1)}}x_{i_{\sigma(m)}},
%$$
$$
r_{i_1i_2\cdots i_m}=\sum_{j=1}^m (-1)^{m-j}r_{i_1\cdots \hat{i}_j \cdots i_{m}}\otimes x_{i_j}\in V^{\otimes m},
$$
for any (not necessarily distinguished) $i_1,i_2,\cdots,i_m\in \{1,2,\cdots,n\}$. The following result is clear.
\begin{lemma}\label{lemma: properties of r_i1i2...} Let an integer $m\geq 2$, (not necessarily distinguished) $i_1,i_2,\cdots,i_m\in \{1,2,\cdots,n\}$. Then
\begin{enumerate}
\item $r_{i_1i_2\cdots i_m}=\sum_{j=1}^m (-1)^{j+1}x_{i_j}\otimes r_{i_1\cdots \hat{i}_j \cdots i_{m}}$.
\item $r_{i_1\cdots i_m}\in W_m$.

\item $r_{i_1\cdots i_m}=0$, if $i_s=i_t$ for some $s\neq t$.

\item $r_{i_1\cdots i_m}=(-1)^{\mathrm{sgn}\sigma}r_{i_{\sigma(1)}\cdots i_{\sigma(m)}}$ for any $\sigma\in S_m$.

\item the set $\{r_{i_1\cdots i_m}\;|\; i_1<i_2<\cdots<i_m\}$ is a basis of $W_m$.
\end{enumerate}
\end{lemma}

Let $\delta:V\to V\otimes V$ be a linear map such that the map $\overline{\delta}$ in Theorem \ref{thm: NA of Ore ext over Polynomial} is induced by it. Write
$$
\delta(x_i)=\sum_{s,t=1}^n k_{st}^{(i)} x_s\otimes x_t,
$$
where $k_{st}^{(i)}\in k$ for $s,t,i=1,\cdots,n$. Then we construct a sequence pair for $\overline{\delta}$.

\begin{lemma}\label{lemma: delta_l delta_r for polynomial} There exists a sequence pair $(\{\delta_{i,r}\},\{\delta_{i,l}\})$ for $\overline{\delta}$  such that
\begin{align*}
\delta_{m,r}(r_{i_1i_2\cdots i_m})=\sum_{j=1}^m\sum_{s,t=1}^n k^{(i_{j})}_{st} r_{i_1\cdots i_{j-1}si_{j+1}\cdots i_m} \otimes x_t,\quad
\delta_{m,l}(r_{i_1i_2\cdots i_m})=\sum_{j=1}^m\sum_{s,t=1}^n k^{(i_{j})}_{st} x_s\otimes r_{i_1\cdots i_{j-1}ti_{j+1}\cdots i_m},
\end{align*}
for any $1\leq i_1<i_2<\cdots<i_m\leq n$ and $m\geq 2$.
\end{lemma}
\begin{proof} By the extension of $\delta$ to $T(V)$, one obtains that for any $1\leq i_1<i_2\leq n$,
\begin{align*}
\delta(r_{i_1i_2})&=\delta(x_{i_1}\otimes x_{i_2}-x_{i_2}\otimes x_{i_1})=x_{i_1}\otimes \delta(x_{i_2})+\delta(x_{i_1})\otimes x_{i_2}-x_{i_2}\otimes \delta(x_{i_1})-\delta(x_{i_2})\otimes x_{i_1}\\
&= \sum_{s,t=1}^n k_{st}^{(i_2)}( x_{i_1}\otimes x_s\otimes x_t- x_s\otimes x_t\otimes x_{i_1})+ \sum_{s,t=1}^n k_{st}^{(i_1)}(x_s\otimes x_t\otimes x_{i_2}-x_{i_2}\otimes  x_s\otimes x_t)\\
&=\sum_{s,t=1}^n \left( k_{st}^{(i_2)} \left(x_{i_1}\otimes x_s\otimes x_t-x_s\otimes x_{i_1}\otimes x_t\right)+k_{st}^{(i_2)} \left(x_s\otimes x_{i_1}\otimes x_t- x_s\otimes x_t\otimes x_{i_1}\right)\right)\\
&\quad+\sum_{s,t=1}^n \left(k_{st}^{(i_1)}\left(x_s\otimes x_t\otimes x_{i_2}-x_s\otimes x_{i_2}\otimes x_t\right)+k_{st}^{(i_1)}\left(x_s\otimes x_{i_2}\otimes x_t-x_{i_2}\otimes  x_s\otimes x_t\right)\right)\\
%&=\sum_{s,t=1}^m \left( k_{st}^{(i_2)}r_{i_1s}x_t +k_{st}^{(i_2)} x_sr_{i_1t}\right)+\sum_{s,t=1}^m \left(k_{st}^{(i_1)}x_sr_{ti_2}+k_{st}^{(i_1)}r_{si_2}x_t\right)\\
&=\sum_{s,t=1}^n x_s\otimes \left( k_{st}^{(i_1)}r_{ti_2}+k_{st}^{(i_2)} r_{i_1t}\right)+\sum_{s,t=1}^n \left(k_{st}^{(i_1)}r_{si_2}+k_{st}^{(i_2)}r_{i_1s} \right)\otimes x_t.
\end{align*}
So we can choose
$$
\delta_{2,r}(r_{i_1i_2})=\sum_{s,t=1}^n \left(k_{st}^{(i_1)}r_{si_2}+k_{st}^{(i_2)}r_{i_1s} \right)\otimes x_t,\qquad \delta_{2,l}(r_{i_1i_2})=\sum_{s,t=1}^n x_s\otimes \left( k_{st}^{(i_1)}r_{ti_2}++k_{st}^{(i_2)} r_{i_1t}\right).
$$

Suppose we have obtained that for any $u<m$,
$$
\delta_{u,r}(r_{i_1i_2\cdots i_u})=\sum_{j=1}^u\sum_{s,t=1}^n k^{(i_{j})}_{st} r_{i_1\cdots i_{j-1}si_{j+1}\cdots i_u} \otimes x_t,\qquad
\delta_{u,l}(r_{i_1i_2\cdots i_u})=\sum_{j=1}^u\sum_{s,t=1}^n k^{(i_{j})}_{st} x_s\otimes r_{i_1\cdots i_{j-1}ti_{j+1}\cdots i_u}.
$$
For any $1\leq i_1<i_2<\cdots<i_m\leq n$, one obtains
\begin{align*}
& (\id_V^{\otimes m-1}\otimes m_B)\left(\id_V^{\otimes m-1}\otimes\delta+\delta_{m-1,r}\otimes \id_V\right)(r_{i_1\cdots i_m})
\\
%=&(\id_V^{\otimes m-1}\otimes m_B)\left(\id_V^{\otimes m-1}\otimes\delta+\delta_{m-1,r}\otimes \id_V\right)\left(\sum_{p=1}^m (-1)^{m-p}r_{i_1\cdots \hat{i}_p \cdots i_{m}}x_{i_p}\right),\\
%=&\sum_{p=1}^m\sum_{s,t=1}^{n}(-1)^{m-p}(\id_V^{\otimes m-1}\otimes m_B)
%\left( k^{(i_p)}_{st}
%r_{i_1\cdots \hat{i}_p \cdots i_{m}}\otimes x_s\otimes x_t+\sum_{j=1}^{p-1} k_{st}^{(i_j)}r_{i_1\cdots i_{j-1}s\cdots \hat{i}_p \cdots i_{m}}\otimes x_t\otimes x_{i_p}\right.\\
%&\left.+\sum_{j=p+1}^{m} k_{st}^{(i_j)}r_{i_1\cdots \hat{i}_p\cdots si_{j+1}\cdots i_{m}}\otimes x_t\otimes x_{i_p}\right)\\
=&\sum_{p=1}^m\sum_{s,t=1}^{n}(-1)^{m-p}(\id_V^{\otimes m-1}\otimes m_B)
\left( k^{(i_p)}_{st}
r_{i_1\cdots \hat{i}_p \cdots i_{m}}\otimes x_s\otimes x_t\right)\\
&+\sum_{p=1}^m\sum_{s,t=1}^{n}(-1)^{m-p}(\id_V^{\otimes m-1}\otimes m_B)\left(\sum_{j=1}^{p-1} k_{st}^{(i_j)}r_{i_1\cdots i_{j-1}s\cdots \hat{i}_p \cdots i_{m}}\otimes x_t\otimes x_{i_p}+\sum_{j=p+1}^{m} k_{st}^{(i_j)}r_{i_1\cdots \hat{i}_p\cdots si_{j+1}\cdots i_{m}}\otimes x_t\otimes x_{i_p}\right)\\
%=&\sum_{s,t=1}^{n}\sum_{p=1}^m(\id_V^{\otimes m-1}\otimes m_B)
%\left((-1)^{m-p} k^{(i_p)}_{st}
%r_{i_1\cdots \hat{i}_p \cdots i_{m}}x_sx_t
%\right)\\
%&+\sum_{s,t=1}^{n}
%\sum_{q=1}^{m} (\id_V^{\otimes m-1}\otimes m_B)\left( \sum_{j=1}^{q-1}(-1)^{m-j}k_{st}^{(i_q)}r_{i_1\cdots \hat{i}_jsi_{q+1}\cdots i_m}x_tx_{i_j}+\sum_{j=q+1}^m (-1)^{m-j}k_{st}^{(i_q)}r_{i_1\cdots i_{q-1}s\cdots \hat{i}_j\cdots i_m}x_tx_{i_j}
%\right)\\
=&\sum_{s,t=1}^{n}\sum_{j=1}^m(\id_V^{\otimes m-1}\otimes m_B)
\left((-1)^{m-j} k^{(i_j)}_{st}
r_{i_1\cdots \hat{i}_j \cdots i_{m}}\otimes x_s\otimes x_t
\right)\\
&+\sum_{s,t=1}^{n}
\sum_{j=1}^{m} (\id_V^{\otimes m-1}\otimes m_B)\left( \sum_{p=1}^{j-1}(-1)^{m-p}k_{st}^{(i_j)}r_{i_1\cdots \hat{i}_p\cdots si_{j+1}\cdots i_m}\otimes x_{i_p}\otimes x_t+\sum_{p=j+1}^m (-1)^{m-p}k_{st}^{(i_j)}r_{i_1\cdots i_{j-1}s\cdots \hat{i}_p\cdots i_m}\otimes x_{i_p}\otimes x_t
\right)\\
=&\sum_{s,t=1}^{n}\sum_{j=1}^m(\id_V^{\otimes m-1}\otimes m_B)
k^{(i_j)}_{st}r_{i_1\cdots i_{j-1}si_{j+1}\cdots i_m}\otimes x_t
\end{align*}
By Remark \ref{remark: choice of delta_r}, we can define
$$
\delta_{u,r}(r_{i_1i_2\cdots i_m})=\sum_{j=1}^m\sum_{s,t=1}^n k^{(i_{j})}_{st} r_{i_1\cdots i_{j-1}si_{j+1}\cdots i_m}\otimes x_t.
$$

Similarly, one obtains the result for $\delta_{i,l}$ by Lemma \ref{lemma: construction of delta_l} and Lemma \ref{lemma: properties of r_i1i2...}(a).
\end{proof}

\begin{proof}[Proof of Theorem \ref{thm: NA of Ore ext over Polynomial}]

By Lemma \ref{lemma: properties of r_i1i2...}, $r_{1\cdots n}$ is the basis of $W_n$. By Lemma \ref{lemma: properties of r_i1i2...} and Lemma \ref{lemma: delta_l delta_r for polynomial}, one obtains
$$
\delta_{n,r}(r_{1\cdots n})=r_{1\cdots n}\otimes \left(\sum^n_{s,t=1}k^{(s)}_{st} x_t\right),\qquad \delta_{n,l}(r_{1\cdots n})=\left(\sum^n_{s,t=1}k^{(t)}_{st} x_s\right)\otimes r_{1\cdots n}.
$$

So $\delta_r=\sum^n_{s,t=1}k^{(s)}_{st} x_t$ and $\delta_l=\sum^n_{s,t=1}k^{(t)}_{st} x_s$, and
$$
\delta_r+\delta_l=\sum^n_{s=1}\left( \sum^n_{t=1} (k^{(s)}_{st}+k^{(s)}_{ts})  x_t\right)
=\sum^n_{s=1} \dfrac{\partial(\delta(x_s))}{\partial x_s}=\nabla\cdot \overline{\delta}.
$$
Since $A$ is Koszul CY, the result follows by Theorem \ref{thm: nakayama automorphism of ore extension}.
\end{proof}

\subsection{Koszul AS-regular algebras of dimension 2}

In this subsection, we assume $\mathrm{char}\, k=0$. We give a formula of Nakayama automorphisms for graded Ore extensions of Koszul AS-regular algebras of dimension 2, and then compute specific Nakayama automorphisms for noetherian ones.

Now we assume $A$ is an AS-regular algebra of dimension 2. By \cite[Theorem 0.1]{Z}, one obtains that $A$ is always Koszul and there is an invertible matrix $Q\in M_n(k)$ such that
$
A\cong k\langle x_1,x_2,\cdots,x_n\rangle/(r),
$
where
$$
r=\mathbf{x}^TQ\mathbf{x},%\ left(\begin{array}{c}x_1\\x_2\\\vdots\\x_n\end{array}\right).
$$
and $\mathbf{x}=(x_1,x_2,\cdots,x_n)^T$.
It is well known (for example, \cite[Scetion 3]{HVZ}) that the Nakayama automorphism of $\mu_A$ satisfies
$$
\mu_A(\mathbf{x})
%\left(\begin{array}{c}x_1\\x_2\\\vdots\\x_n\end{array}\right)
=-(Q^{-1})^TQ\mathbf{x}.
%\left(\begin{array}{c}x_1\\x_2\\\vdots\\x_n\end{array}\right).
$$

Let $\overline{\sigma}$ be a graded automorphism of $A$ and $\overline{\delta}$ a degree-one $\overline{\sigma}$-derivation of $A$. Write $M\in M_n(k)$ for the invertible matrix such that $\overline{\sigma}(\mathbf{x})=M\mathbf{x}$.  Choose a linear map $\delta:V\to V\otimes V$ such that $\overline{\delta}$ can be induced by it, where $V$ is the vector space spanned by $\{x_1,\cdots,x_n\}$. Since $A$ is 2-dimensional, there is a unique pair $(\delta_r,\delta_l)$ of elements in $V$ such that
$$
\delta(r)=r\otimes \delta_r+\delta_l\otimes r.
$$
That is, there are two certain elements $c_r=(c_{r1},c_{r2},\cdots,c_{rn}),c_l=(c_{l1},c_{l2},\cdots,c_{ln})\in k^n$ such that
$$
\delta_r=c_r\mathbf{x},\qquad \delta_l=c_l\mathbf{x}.
$$

By Theorem \ref{thm: nakayama automorphism of ore extension}, the Nakayama automorphism $\mu_B$ of the graded Ore extension $B=A[z;\overline{\sigma},\overline{\delta}]$ satisfies
%$$
%{\mu_B}_{\mid A}=\overline{\sigma}^{-1}\mu_A,\qquad \mu_B(z)=\hdet\overline{\sigma}+\nabla_{\overline{\sigma}}\cdot\overline{\delta}.
%$$
\begin{equation}\label{nakayama of comm}
{\mu_B}\left(\begin{array}{c}\mathbf{x}
%x_1\\\vdots\\x_n
\\z\end{array}\right)
=\left(
\begin{array}{cc}
-M^{-1}(Q^T)^{-1}Q &0\\
c_r- c_l M^{-1}(Q^T)^{-1}Q    &\hdet(\overline{\sigma})
\end{array}
\right)
\left(\begin{array}{c}\mathbf{x}
%x_1\\\vdots\\x_n
\\z\end{array}\right).
\end{equation}

Now we focus on noetherian ones. There is an interesting result about CY property for noetherian cases.
\begin{theorem}\label{thm: CY for 2-dim}
Let $A=k\langle x_1,x_2\rangle/(r)$ be a noetherian AS-regular algebra of dimension 2 and $B=A[z;\overline{\sigma},\overline{\delta}]$ is a graded Ore extension. Write $\mu_A$ for the Nakayama automorphism of $A$.
\begin{enumerate}
\item  Suppose $A$ is commutative, then $B$ is CY if and only if $\overline{\sigma}=\id_A$ and
    $$
    \overline{\delta}(x_1)=l_1x_1^2-2l_4x_2x_1+l_2x_2^2,\qquad
    \overline{\delta}(x_2)=l_3x_1^2-2l_1x_2x_1+l_4x_2^2,
    $$
for some $l_1,l_2,l_3,l_4\in k$.
\item  Suppose $A$ is noncommutative, then $B$ is CY if and only if $\overline{\sigma}=\mu_A$.
\end{enumerate}
\end{theorem}

To prove this result, we determine all graded Ore extensions of noetherian Koszul AS-regular algebras of dimension $2$ and compute their Nakayama automorphisms. Let $A$ be a noetherian Koszul AS-regular algebra of dimension $2$. Then $A=k\langle x_1,x_2\rangle/(r)$, and there are only two classes of $A$ up to isomorphism, that is,
$$
Q=\left(
\begin{array}{cc}
0 & 1\\
-q  &  0
\end{array}
\right),\qquad\text{or}\qquad
Q=\left(
\begin{array}{cc}
0 & 1\\
-1  &  -1
\end{array}
\right),
$$
where $q$ is a nonzero element in $k$. Let $V$ be the vector space spanned by $\{x_1,x_2\}$.

In the following, $\overline{\sigma}$ is a graded automorphism of $A$, $M=(m_{ij})\in M_2(k)$ is the invertible matrix such that $\overline{\sigma}(x_1,x_2)^T=M(x_1,x_2)^T$. Write $\sigma=\overline{\sigma}_{\mid V}$ and $\sigma_T=\oplus_{i\geq0}\sigma^{\otimes i}$. Let $\delta$ be a linear map from $V$ to $V\otimes V$. Then any degree-one $\overline{\sigma}$-derivation $\overline{\delta}$ of $A$ can be induced by $\delta$, in case $\delta$ extends to a degree-one  $\sigma_T$-derivation of $T(V)(\cong k\langle x_1,x_2\rangle)$ such that
\begin{equation}\label{condtion for comm}
\delta(r)\in r\otimes V+V\otimes r.
\end{equation}
Since the forms of $r$ (or $Q$), we assume without loss of generality,
$$
\delta(x_i)=\gamma_{i1}x_1^2+\gamma_{i2}x_2x_1+\gamma_{i3}x_2^2,
$$
where $\gamma_{ij}\in k$ for $j=1,2,3,i=1,2.$

\subsubsection{Case (i): commutative polynomial} In this case, $r=x_1x_2-x_2x_1$, or equivalently
$$Q=\left(
\begin{array}{cc}
0 & 1\\
-1  &  0
\end{array}
\right),$$
$M=(m_{ij})$ is an arbitrary invertible matrix, and $\hdet(\overline{\sigma})=\det (M)$ (We refer \cite{JZ,SZL} for the computation of homological determinant). We have
%Let $\delta$ be a linear map from $V$ to $V\otimes V$. Then any $\overline{\sigma}$-derivation $\overline{\delta}$ of $A$ can be induced by $\delta$, in case $\delta$ extends to a degree-one  $\sigma_T$-derivation of $T(V)$ such that
%\begin{equation}
%\delta(r)\in r\otimes V+V\otimes r,
%\end{equation}
%where $r=x_1x_2-x_2x_1$. Without loss of generality, we assume $$
%\delta(x_i)=\gamma_{i1}x_1^2+\gamma_{i2}x_2x_1+\gamma_{i3}x_2^2,
%$$
%where $\gamma_{ij}\in k$ for $j=1,2,3,i=1,2.$ Then
\begin{align*}
\delta(r)=&(m_{11}\gamma_{21}-m_{21}\gamma_{11}-\gamma_{21})x_1^3+(m_{12}\gamma_{23}-m_{22}\gamma_{13}+\gamma_{13})x_2^3+(m_{11}\gamma_{22}-m_{21}\gamma_{12})x_1x_2x_1+\gamma_{12}x_2x_1x_2\\
&+\gamma_{11}x_1^2x_2+(m_{11}\gamma_{23}-m_{21}\gamma_{13})x_1x_2^2+(m_{12}\gamma_{21}-m_{22}\gamma_{11}-\gamma_{22})x_2x_1^2+(m_{12}\gamma_{22}-m_{22}\gamma_{12}-\gamma_{23})x_2^2x_1.
\end{align*}

By a straightforward computation, one obtains that the following result.
\begin{lemma}
The condition (\ref{condtion for comm}) is equivalent to the following equations hold
\begin{equation}\label{equation for comm}
\begin{aligned}
&m_{21}\gamma_{11}+(1-m_{11})\gamma_{21}=0,\\
&(1-m_{22})\gamma_{13}+m_{12}\gamma_{23}=0,\\
&(m_{22}-1)\gamma_{11}+m_{21}\gamma_{12}-m_{12}\gamma_{21}+(1-m_{11})\gamma_{22}=0,\\
&(m_{22}-1)\gamma_{12}+m_{21}\gamma_{13}-m_{12}\gamma_{22}+(1-m_{11})\gamma_{23}=0.
\end{aligned}
\end{equation}
In this case, $\delta(r)=r\otimes \delta_r+\delta_l\otimes r$, where
\begin{equation*}\label{delta_l,r for comm}
\delta_{r}=(m_{22}\gamma_{11}-m_{12}\gamma_{21}+\gamma_{22})x_1+(m_{11}\gamma_{23}-m_{21}\gamma_{13})x_2,\quad
\delta_{l}=\gamma_{11}x_1+(m_{22}\gamma_{12}-m_{12}\gamma_{22}+\gamma_{23})x_2.
\end{equation*}
\end{lemma}

By (\ref{nakayama of comm}), the Nakayama automorphisms of graded Ore extension $B=A[z;\overline{\sigma},\overline{\delta}]$ of $A$ by (\ref{nakayama of comm}), where $\overline{\delta}$ is induced by $\delta$, satisfies
\begin{equation}\label{naka equa for comm}
\begin{aligned}
&\mu_B(x_1)=\det(M)^{-1}(m_{22}x_1-m_{12}x_2),\quad \mu_B(x_2)=\det(M)^{-1}(-m_{21}x_1+m_{11}x_2),\\ &\mu_B(z)=\det(M)z+(x_1,x_2)\upsilon^T,
    \end{aligned}
\end{equation}
where $\upsilon=\left(m_{22}\gamma_{11}-m_{12}\gamma_{21}+\gamma_{22},m_{11}\gamma_{23}-m_{21}\gamma_{13}\right)
+\left(\gamma_{11},m_{22}\gamma_{12}-m_{12}\gamma_{22}+\gamma_{23}\right)M^{-1}$. Then we list all solutions of equations (\ref{equation for comm}).

%\delta\left(
%\begin{array}{c}
%x_1\\
%x_2
%\end{array}
%\right)
%=\left(
%\begin{array}{ccc}
%x_1\\
%x_2
%\end{array}
%\right)
%\left(
%\begin{array}{c}
%x_1^2\\
%x_2x_1\\
%x_2x
%\end{array}
%\right)
%$$
\begin{solution}\label{solution: commutative}
All soloutions of  equations (\ref{equation for comm}) are as follows.
\begin{enumerate}
\item If $M$ is the identity matrix $E_2$, that is $m_{11}=m_{22}=1$ and $m_{12}=m_{21}=0$, then each $\gamma_{ij}$ is free for $i=1,2,j=1,2,3$;% and
%    \begin{align*}
%    \mu_B(x_1)=x_1,\quad \mu_B(x_2)=x_2,\quad \mu_B(z)=z+(2\gamma_{11}+\gamma_{22})x_1+(\gamma_{12}+2\gamma_{23})x_2.
%    \end{align*}
%It is a special case of Theorem \ref{thm: NA of Ore ext over Polynomial}.
\item If $m_{21}=0$, $m_{22}=1$ and $M\neq E_2$, then  $
    \gamma_{21}=\gamma_{22}=\gamma_{23}=0$ and $ \gamma_{11},\gamma_{12},\gamma_{13}$  are free variables; % and
%    $$\qquad
%    \mu_B(x_1)=m_{11}^{-1}x_1-m_{11}^{-1}m_{12}x_2,\quad \mu_B(x_2)=x_2,\quad \mu_B(z)=m_{11}z+(1+m_{11}^{-1})\gamma_{11}x_1-(m_{11}^{-1}m_{12}\gamma_{11}-\gamma_{12})x_2.
%    $$

\item If $m_{21}=0,m_{22}\neq 1,m_{11}=1$, then
 $\gamma_{11}=(m_{22}-1)^{-1}m_{12}\gamma_{21},
    \gamma_{12}=(m_{22}-1)^{-1}m_{12}\gamma_{22},
    \gamma_{13}=(m_{22}-1)^{-1}m_{12}\gamma_{23};$
 %$\gamma_{11}=\dfrac{m_{12}}{m_{22}-1}\gamma_{21},
  %  \gamma_{12}=\dfrac{m_{12}}{m_{22}-1}\gamma_{22},
  %  \gamma_{13}=\dfrac{m_{12}}{m_{22}-1}\gamma_{23};$
 %   \begin{align*}
  %  &\gamma_{11}=\frac{m_{12}}{m_{22}-1}\gamma_{21},\quad
   % \gamma_{12}=\frac{m_{12}}{m_{22}-1}\gamma_{22},\quad
    %\gamma_{13}=\frac{m_{12}}{m_{22}-1}\gamma_{23};%\\
   % &\mu_B(x_1)=x_1-m_{12}m_{22}^{-1}x_2,\quad \mu_B(x_2)=m_{22}^{-1}x_2,\\ &\mu_B(z)=m_{22}z+\left(\frac{2m_{22}}{m_{22}-1}\gamma_{21}+\gamma_{22}\right)x_1-\left(\frac{m_{12}^2}{m_{22}(m_{22}-1)}\gamma_{21}-\frac{m_{12}}{m_{22}(m_{22}-1)}\gamma_{22}-(1+m_{22}^{-1})\gamma_{23}\right)x_2.
  %  \end{align*}
% $$
%    \gamma_{11}=\frac{m_{12}}{m_{22}-1}\gamma_{21},\qquad
%    \gamma_{12}=\frac{m_{12}}{m_{22}-1}\gamma_{22},\qquad
%    \gamma_{13}=\frac{m_{12}}{m_{22}-1}\gamma_{23},
%    $$

\item If $m_{21}=0,m_{22}\neq 1,m_{11}\neq1$, then $\gamma_{11}=(m_{22}-1)^{-1}(m_{11}-1)\gamma_{22},
    \gamma_{12}=(m_{22}-1)^{-1}(m_{12}\gamma_{22}+(m_{11}-1)\gamma_{23}),
    \gamma_{13}=(m_{22}-1)^{-1}m_{12}\gamma_{23},
    \gamma_{21}=0;$%
%    $\gamma_{11}=\dfrac{m_{11}-1}{m_{22}-1}\gamma_{22},
%    \gamma_{12}=\dfrac{m_{12}}{m_{22}-1}\gamma_{22}+\dfrac{m_{11}-1}{m_{22}-1}\gamma_{23},
%    \gamma_{13}=\dfrac{m_{12}}{m_{22}-1}\gamma_{23},
%    \gamma_{21}=0;$
%    \begin{align*}
%    &\gamma_{11}=\frac{m_{11}-1}{m_{22}-1}\gamma_{22},\quad
%    \gamma_{12}=\frac{m_{12}}{m_{22}-1}\gamma_{22}+\frac{m_{11}-1}{m_{22}-1}\gamma_{23},\quad
%    \gamma_{13}=\frac{m_{12}}{m_{22}-1}\gamma_{23},\quad
%    \gamma_{21}=0,\\
%    &\mu_B(x_1)=m_{11}^{-1}x_1-(m_{11}m_{22})^{-1}m_{12}x_2,\quad \mu_B(x_2)=m_{22}^{-1}x_2,\\ &\mu_B(z)=m_{11}m_{22}z+\frac{m_{11}m_{22}-m_{11}^{-1}}{m_{22}-1}\gamma_{22}x_1+\left(\frac{m_{12}}{m_{11}m_{22}(m_{22}-1)}\gamma_{22}+\frac{m_{11}m_{22}-m_{22}^{-1}}{m_{22}-1}\gamma_{23}\right)x_2.
%    \end{align*}

\item If $m_{12}=0$, there exit symmetric solutions of the last three ones, which we omit here;

\item If $m_{21}m_{22}\neq 0$ and $(m_{22}-1)(m_{11}-1)=m_{21}m_{12}$, then
    $
    \gamma_{11}=m_{21}^{-1}(m_{11}-1)\gamma_{21},
    \gamma_{12}=m_{21}^{-1}(m_{11}-1)\gamma_{22},
    \gamma_{23}=m_{12}^{-1}(m_{22}-1)\gamma_{13};
    $
%    $
%    \gamma_{11}=\dfrac{m_{11}-1}{m_{21}}\gamma_{21},
%    \gamma_{12}=\dfrac{m_{11}-1}{m_{21}}\gamma_{22},
%    \gamma_{23}=\dfrac{m_{22}-1}{m_{12}}\gamma_{13};
%    $
%  \begin{align*}
%    &\gamma_{11}=\frac{m_{11}-1}{m_{21}}\gamma_{21},\quad
%    \gamma_{12}=\frac{m_{11}-1}{m_{21}}\gamma_{22},\quad
%    \gamma_{23}=\frac{m_{22}-1}{m_{12}}\gamma_{13},\\
%    &\mu_B(x_1)=\det(M)^{-1}(m_{22}x_1-m_{12}x_2),\quad \mu_B(x_2)=\det(M)^{-1}(-m_{21}x_1+m_{22}x_2),\\ &\mu_B(z)=\det(M)z+(x_1,x_2)\upsilon_1^T,
%    \end{align*}
%where $\upsilon_1=\left(\frac{m_{11}-1}{m_{21}}\gamma_{21}+\gamma_{22},\frac{m_{22}-1}{m_{11}}\gamma_{13}\right)+\left(\frac{m_{11}-1}{m_{21}}\gamma_{21},\frac{m_{11}-1}{m_{21}}\gamma_{22}+\frac{m_{22}-1}{m_{12}}\gamma_{13}\right)M^{-1}$.
\item If $m_{21}m_{22}\neq 0$ and $(m_{22}-1)(m_{11}-1)\neq m_{21}m_{12}$,
then $
\gamma_{11}=m_{21}^{-1}(m_{11}-1)\gamma_{21},
    \gamma_{12}=m_{21}^{-1}m_{12}\gamma_{21}+m_{12}^{-1}(m_{11}-1)\gamma_{13},
    \gamma_{22}=m_{21}^{-1}(m_{22}-1)\gamma_{21}+m_{12}^{-1}m_{21}\gamma_{13},
    \gamma_{23}=m_{12}^{-1}(m_{22}-1)\gamma_{13}.
$
\end{enumerate}
\end{solution}

\subsubsection{Case (ii): noncommutative quantum plane
}In this case, $r=x_1x_2-qx_2x_1$, or equivalently,
$$Q=\left(
\begin{array}{cc}
0 & 1\\
-q  &  0
\end{array}
\right),$$
where  nonzero element $q\neq 1$. There are two subcases: $q=-1$ and $q\neq 1$.

\noindent (1) $q=-1$. The graded automorphisms of $A=k\langle x_1,x_2\rangle/(x_1x_2+x_2x_1)$ have two forms
$$
M=
\left(
\begin{array}{cc}
m_{11}  & 0\\
0   &m_{22}
\end{array}
\right)\qquad\text{and}\qquad
M=
\left(
\begin{array}{cc}
0  & m_{12}\\
m_{21}   &0
\end{array}
\right).
$$

\ding{172} $ M=
\left(
\begin{array}{cc}
m_{11}  & 0\\
0   &m_{22}
\end{array}
\right)$. Then $\hdet(\overline{\sigma})=m_{11}m_{22}$. Since $r=x_1x_2+x_2x_1$, then
\begin{align*}
\delta(r)=&(m_{11}+1)\gamma_{21}x_1^3+(m_{22}+1)\gamma_{13}x_2^3+m_{11}\gamma_{22}x_1x_2x_1+\gamma_{12}x_2x_1x_2\\
&+m_{11}\gamma_{23}x_1x_2^2+\gamma_{11}x_1^2x_2+(m_{22}\gamma_{11}+\gamma_{22})x_2x_1^2+(m_{22}\gamma_{12}+\gamma_{23})x_2^2x_1.
\end{align*}
The following result is easy to get.
\begin{lemma}
The condition (\ref{condtion for comm}) is equivalent to the following equations hold
\begin{equation}\label{equation for q=-1}
\begin{aligned}
&(m_{11}+1)\gamma_{21}=0,\\
&(m_{22}+1)\gamma_{13}=0,\\
&(m_{22}+1)\gamma_{11}+(1-m_{11})\gamma_{22}=0,\\
&(m_{22}-1)\gamma_{12}+(m_{11}+1)\gamma_{23}=0.
\end{aligned}
\end{equation}
In this case, $\delta(r)=r\otimes \delta_r+\delta_l\otimes r$, where
\begin{equation*}\label{delta_l,r for q=-1}
\delta_{r}=(m_{22}\gamma_{11}+\gamma_{22})x_1+m_{11}\gamma_{23}x_2,\quad
\delta_{l}=\gamma_{11}x_1+(m_{22}\gamma_{12}+\gamma_{23})x_2.
\end{equation*}
\end{lemma}
By (\ref{nakayama of comm}), we have the Nakayama automorphism of a graded Ore extension $B=A[z;\overline{\sigma},\overline{\delta}]$ satisfies that
 \begin{equation}\label{naka equa for q=-1}
 \begin{aligned}
    &\mu_B(x_1)=-m_{11}^{-1}x_1,\quad \mu_B(x_2)=-m_{22}^{-1}x_2,\\ &\mu_B(z)=m_{11}m_{22}z+((m_{22}-m_{11}^{-1})\gamma_{11}+\gamma_{22})x_1+((m_{11}-m_{22}^{-1})\gamma_{23}-\gamma_{12})x_2.
    \end{aligned}
 \end{equation}
To be explicit, we give all solutions of (\ref{equation for q=-1}).
\begin{solution}\label{solution: q=-1}
The solutions of (\ref{equation for q=-1}) are as follows.
\begin{enumerate}
\item If $m_{11}=m_{22}=-1$, then $\gamma_{12}=\gamma_{22}=0$ and the other variables are free;%, and
%     \begin{align*}
%    &\mu_B(x_1)=x_1,\quad \mu_B(x_2)=x_2,\quad \mu_B(z)=z.
%    \end{align*}

\item If $m_{11}=1,m_{22}=-1$, then $\gamma_{21}=0,\gamma_{23}=\gamma_{12}$ and  $\gamma_{11},\gamma_{13},\gamma_{22}$ are free;%, and
%     \begin{align*}
%    &\mu_B(x_1)=-x_1,\quad \mu_B(x_2)=x_2,\quad \mu_B(z)=-z+(\gamma_{22}-2\gamma_{11})x_1+\gamma_{12}x_2.
%    \end{align*}

\item If $m_{11}\neq \pm1,m_{22}=-1$, then $\gamma_{21}=\gamma_{22}=0$, $\gamma_{23}=2(m_{11}+1)^{-1}\gamma_{12}$ and $\gamma_{11},\gamma_{13}$ are free;%, and
%     \begin{align*}
%    &\mu_B(x_1)=-m_{11}^{-1}x_1,\quad \mu_B(x_2)=x_2,\quad \mu_B(z)=-m_{11}z-(1+m_{11}^{-1})\gamma_{11}x_1+\gamma_{12}x_2.
%    \end{align*}

\item If $m_{11}=-1,m_{22}=1$, then $\gamma_{13}=0,\gamma_{11}=-\gamma_{22}$ and  $\gamma_{12},\gamma_{21},\gamma_{23}$ are free;%, and
   %  \begin{align*}
%    &\mu_B(x_1)=x_1,\quad \mu_B(x_2)=-x_2,\quad \mu_B(z)=-z-\gamma_{22}x_1-(2\gamma_{23}+\gamma_{12})x_2.
%    \end{align*}

%\item If $m_{11}=1,m_{22}=1$, then $\gamma_{11}=\gamma_{13}=\gamma_{21}=\gamma_{23}=0$,  and the other variables are free, and
%     \begin{align*}
%    &\mu_B(x_1)=-x_1,\quad \mu_B(x_2)=-x_2,\quad \mu_B(z)=z+\gamma_{22}x_1-\gamma_{12}x_2.
%    \end{align*}
%
%\item If $m_{11}\neq-1,m_{22}=1$, then $\gamma_{13}=\gamma_{21}=\gamma_{23}=0,\gamma_{11}=\dfrac{m_{11}-1}{2}\gamma_{22}$,  and the other variables are free, and
%     \begin{align*}
%    &\mu_B(x_1)=-m_{11}^{-1}x_1,\quad \mu_B(x_2)=-x_2,\quad \mu_B(z)=m_{11}z+\frac{m_{11}+m_{11}^{-1}}{2}\gamma_{22}x_1-\gamma_{12}x_2.
%    \end{align*}
\item  If $m_{11}=-1,m_{22}\neq \pm 1$, then $\gamma_{12}=\gamma_{13}=0,\gamma_{11}=-2(m_{22}+1)^{-1}\gamma_{22}$ and $\gamma_{21},\gamma_{23}$ are free;%, and
%     \begin{align*}
%    &\mu_B(x_1)=x_1,\quad \mu_B(x_2)=-m_{22}^{-1}x_2,\quad \mu_B(z)=-m_{22}z-\gamma_{22}x_1-(1+m_{22}^{-1})\gamma_{23}x_2.
%    \end{align*}

    \item  If $m_{11}\neq-1,m_{22}\neq- 1$, then $\gamma_{13}=\gamma_{21}=0,\gamma_{11}=(m_{11}-1)(m_{22}+1)^{-1}\gamma_{22}, \gamma_{23}=(m_{11}+1)^{-1}(1-m_{22})\gamma_{12}$.  %and
%     \begin{align*}
%    &\mu_B(x_1)=-m_{11}^{-1}x_1,\quad \mu_B(x_2)=-m_{22}^{-1}x_2,\\ &\mu_B(z)=m_{11}m_{22}z+(m_{11}m_{22}+m_{11}^{-1})(m_{22}+1)^{-1}\gamma_{22}x_1-(m_{11}m_{22}+m_{22}^{-1})(m_{11}+1)^{-1}\gamma_{12}x_2.
%    \end{align*}

\end{enumerate}
\end{solution}

\ding{173} $M=
\left(
\begin{array}{cc}
0  & m_{12}\\
m_{21}   &0
\end{array}
\right).$  Then $\hdet(\overline{\sigma})=m_{12}m_{21}$. Similarly, one obtains that
\begin{align*}
\delta(r)=&(m_{21}\gamma_{11}+\gamma_{21})x_1^3+(m_{12}\gamma_{23}+\gamma_{13})x_2^3+m_{12}\gamma_{12}x_1x_2x_1+\gamma_{12}x_2x_1x_2\\
&+m_{21}\gamma_{13}x_1x_2^2+\gamma_{11}x_1^2x_2+(m_{12}\gamma_{21}+\gamma_{22})x_2x_1^2+(m_{12}\gamma_{22}+\gamma_{23})x_2^2x_1.
\end{align*}
\begin{lemma}
The condition (\ref{condtion for comm}) is equivalent to the following equations hold
\begin{equation}\label{equation for q=-1 II}
\begin{aligned}
&m_{21}\gamma_{11}+\gamma_{21}=0,\\
&m_{12}\gamma_{23}+\gamma_{13}=0,\\
&\gamma_{11}-m_{21}\gamma_{12}+m_{12}\gamma_{21}+\gamma_{22}=0,\\
&-\gamma_{12}+m_{21}\gamma_{13}+m_{12}\gamma_{22}+\gamma_{23}=0.
\end{aligned}
\end{equation}
In this case, $\delta(r)=r\otimes \delta_r+\delta_l\otimes r$, where
\begin{equation*}\label{delta_l,r for qneq-1}
\delta_{r}=(m_{12}\gamma_{21}+\gamma_{22})x_1+m_{21}\gamma_{13}x_2,\quad
\delta_{l}=\gamma_{11}x_1+(m_{12}\gamma_{22}+\gamma_{23})x_2.
\end{equation*}
\end{lemma}

%The solutions to equations (\ref{equation for q=-1 II}) and the Nakayama automorphism of a graded Ore extension $B=A[z;\overline{\sigma},\overline{\delta}]$ are as follows.
\begin{solution}\label{solution: q=-1 II}
The solutions to equations (\ref{equation for q=-1 II}) are as follows.
\begin{enumerate}
\item If $m_{12}m_{21}=1$, then
$
\gamma_{11}=-m_{12}\gamma_{21},\gamma_{12}=m_{12}\gamma_{22},
\gamma_{13}=-m_{12}\gamma_{23};
$
\item If $m_{12}m_{21}\neq1$, then $
\gamma_{11}=-m_{21}^{-1}\gamma_{21},\gamma_{12}=m_{12}m_{21}^{-1}\gamma_{21}+\gamma_{23},
\gamma_{13}=-m_{12}\gamma_{23},\gamma_{22}=m_{21}^{-1}\gamma_{21}+m_{21}\gamma_{23}.
$
Moreover, the Nakayama automorphism of $B=A[z;\overline{\sigma},\overline{\delta}]$ satisfies
\begin{align*}
&\mu_B(x_1)=-m_{21}^{-1}x_2,\quad \mu_B(x_2)=-m_{12}^{-1}x_1,\\ &\mu_B(z)=m_{12}m_{21}z+(m_{12}\gamma_{21}-m_{12}^{-1}\gamma_{23})x_1+(m_{12}^{-2}\gamma_{21}-m_{12}m_{21}\gamma_{23})x_2.
\end{align*}
\end{enumerate}
\end{solution}
\noindent (2) $q\neq-1$. The graded automorphism $\overline{\sigma}$ of $A=k\langle x_1,x_2\rangle/(x_1x_2-qx_2x_1)$ must be the following form
$
M=
\left(
\begin{array}{cc}
m_{11}  & 0\\
0   &m_{22}
\end{array}
\right),
$ and $\hdet(\overline{\sigma})=m_{11}m_{22}$.
One obtains that
\begin{align*}
\delta(r)=&(m_{11}-q)\gamma_{21}x_1^3+(1-qm_{22})\gamma_{13}x_2^3+m_{11}\gamma_{22}x_1x_2x_1+\gamma_{12}x_2x_1x_2\\
&+m_{11}\gamma_{23}x_1x_2^2+\gamma_{11}x_1^2x_2-q(m_{22}\gamma_{11}+\gamma_{22})x_2x_1^2-q(m_{22}\gamma_{12}+\gamma_{23})x_2^2x_1.
\end{align*}
\begin{lemma}
The condition (\ref{condtion for comm}) is equivalent to the following equations hold
\begin{equation}\label{equation for qneq-1}
\begin{aligned}
&(m_{11}-q)\gamma_{21}=0,\\
&(1-qm_{22})\gamma_{13}=0,\\
&(m_{22}-q)\gamma_{11}+(1-m_{11})\gamma_{22}=0,\\
&(m_{22}-1)\gamma_{12}+(1-qm_{11})\gamma_{23}=0.
\end{aligned}
\end{equation}
In this case, $\delta(r)=r\otimes \delta_r+\delta_l\otimes r$, where
\begin{equation*}
\delta_{r}=(m_{22}\gamma_{11}+\gamma_{22})x_1+m_{11}\gamma_{23}x_2,\quad
\delta_{l}=\gamma_{11}x_1+(m_{22}\gamma_{12}+\gamma_{23})x_2.
\end{equation*}
\end{lemma}

By (\ref{nakayama of comm}), we have the Nakayama automorphism of the graded Ore extension $B=A[z;\overline{\sigma},\overline{\delta}]$ satisfies that
 \begin{equation}\label{naka equa for qneq-1}
 \begin{aligned}
    &\mu_B(x_1)=qm_{11}^{-1}x_1,\quad \mu_B(x_2)=(qm_{22})^{-1}x_2,\\ &\mu_B(z)=m_{11}m_{22}z+((m_{22}+qm_{11}^{-1})\gamma_{11}+\gamma_{22})x_1+((m_{11}+(qm_{22})^{-1})\gamma_{23}+q^{-1}\gamma_{12})x_2.
    \end{aligned}
    \end{equation}

\begin{solution}\label{solution: qneq-1}
All solutions of (\ref{equation for qneq-1}) are as follows.
\begin{enumerate}
\item If $m_{11}=q,m_{22}=q^{-1}$, then $\gamma_{12}=-q(1+q)\gamma_{23},\gamma_{22}=-(q^{-1}+1)\gamma_{11}$ and $\gamma_{13},\gamma_{21}$ are free;%, and
 %    \begin{align*}
%    &\mu_B(x_1)=x_1,\quad \mu_B(x_2)=x_2,\quad \mu_B(z)=z.
%    \end{align*}

\item If $m_{11}\neq q,m_{22}=q^{-1}$, then $\gamma_{21}=0,\gamma_{11}=(q^{-1}-q)^{-1}(m_{11}-1)\gamma_{22},\gamma_{12}=(q^{-1}-1)^{-1}(qm_{11}-1)\gamma_{23}$, and $\gamma_{13}$ is free;%, and
%     \begin{align*}
%    &\mu_B(x_1)=qm_{11}^{-1}x_1,\quad \mu_B(x_2)=x_2,\\ &\mu_B(z)=q^{-1}m_{11}z+(q^{-1}-q)^{-1}(q^{-1}m_{11}-qm_{11}^{-1})\gamma_{22}x_1+(1-q)^{-1}(m_{11}-q)\gamma_{23}x_2.
%    \end{align*}
\item If $m_{11}=q,m_{22}\neq q^{-1}$, then $\gamma_{13}=0$, $\gamma_{22}=(1-q)^{-1}(q-m_{22})\gamma_{11}$, $\gamma_{23}=(1-q^2)^{-1}(1-m_{22})\gamma_{12}$ and $\gamma_{21}$ is free;%, and
 %    \begin{align*}
%    &\mu_B(x_1)=x_1,\quad \mu_B(x_2)=(qm_{22})^{-1}x_2,\\ &\mu_B(z)=qm_{22}z+(1-q)^{-1}(1-qm_{22})\gamma_{11}x_1+(1-q^2)^{-1}((qm_{22})^{-1}-qm_{22})\gamma_{12}x_2.
%    \end{align*}

\item If $m_{11}=1,m_{22}\neq q^{-1}$, then $\gamma_{11}=\gamma_{13}=\gamma_{21}=0$, $\gamma_{23}=(1-q)^{-1}(1-m_{22})\gamma_{12}$ and $\gamma_{22}$ is free;%, and
%     \begin{align*}
%    &\mu_B(x_1)=qx_1,\quad \mu_B(x_2)=(qm_{22})^{-1}x_2,\quad \mu_B(z)=m_{22}z+\gamma_{22}x_1+(1-q)^{-1}((qm_{22})^{-1}-m_{22})\gamma_{12}x_2.
%    \end{align*}

\item If $m_{11}=q^{-1},m_{22}\neq q^{-1}$, then $\gamma_{13}=\gamma_{21}=0,\gamma_{22}=(1-q^{-1})^{-1}(q-m_{22})\gamma_{11}$, $\gamma_{23}$ is free and $\gamma_{12}=0$ if $m_{22}\neq 1$ or $\gamma_{12}$ is free if $m_{22}=1$;%, and
%     \begin{align*}
%    &\mu_B(x_1)=q^2x_1,\quad \mu_B(x_2)=(qm_{22})^{-1}x_2,\\ &\mu_B(z)=q^{-1}m_{22}z+(1-q^{-1})^{-1}(q^2-q^{-1}m_{22})\gamma_{11}x_1+q^{-1}((1+m_{22}^{-1})\gamma_{23}+\gamma_{12})x_2.
%    \end{align*}

%\item If $m_{11}=1,m_{22}=1$, then $\gamma_{11}=\gamma_{13}=\gamma_{21}=\gamma_{23}=0$,  and the other variables are free, and
%     \begin{align*}
%    &\mu_B(x_1)=-x_1,\quad \mu_B(x_2)=-x_2,\quad \mu_B(z)=z+\gamma_{22}x_1-\gamma_{12}x_2.
%    \end{align*}
%
%\item If $m_{11}\neq-1,m_{22}=1$, then $\gamma_{13}=\gamma_{21}=\gamma_{23}=0,\gamma_{11}=\dfrac{m_{11}-1}{2}\gamma_{22}$,  and the other variables are free, and
%     \begin{align*}
%    &\mu_B(x_1)=-m_{11}^{-1}x_1,\quad \mu_B(x_2)=-x_2,\quad \mu_B(z)=m_{11}z+\frac{m_{11}+m_{11}^{-1}}{2}\gamma_{22}x_1-\gamma_{12}x_2.
%    \end{align*}
\item  If $m_{11}\neq q^{\pm 1},1,m_{22}\neq q^{-1}$, then $\gamma_{13}=\gamma_{21}=0,\gamma_{22}=(1-m_{11})^{-1}(q-m_{22})\gamma_{11}$,  and $\gamma_{23}=(1-qm_{11})^{-1}(1-m_{22})\gamma_{12}$.% %    and
%     \begin{align*}
%    &\mu_B(x_1)=qm^{-1}_{11}x_1,\quad \mu_B(x_2)=(qm_{22})^{-1}x_2,\\ &\mu_B(z)=m_{11}m_{22}z+(1-m_{11})^{-1}(qm_{11}^{-1}-m_{11}m_{22})x_1+(1-qm_{11})^{-1}((qm_{22})^{-1}-m_{11}m_{22})\gamma_{23}x_2.
%    \end{align*}

\end{enumerate}
\end{solution}

\subsubsection{Case (iii): Jordan plane}In this case, $r=x_1x_2-x_2x_1-x_2^2$, or equivalently,
$$Q=\left(
\begin{array}{cc}
0 & 1\\
-1  &  -1
\end{array}
\right),$$
$M$ has the form $
\left(
\begin{array}{cc}
m_{11}  & m_{12}\\
0   &m_{11}
\end{array}
\right)
$ and $\hdet(\overline{\sigma})=m_{11}^2$. One obtains that
\begin{align*}
\delta(r)=&(m_{11}-1)\gamma_{21}x_1^3+m_{11}\gamma_{22}x_1x_2x_1+(\gamma_{12}-\gamma_{22})x_2x_1x_2+m_{11}\gamma_{23}x_1x_2^2+(\gamma_{11}-\gamma_{21})x_1^2x_2\\
&+((m_{12}-m_{11})\gamma_{21}-m_{11}\gamma_{11}-\gamma_{22})x_2x_1^2+((m_{12}-m_{11})\gamma_{22}-m_{11}\gamma_{12}-\gamma_{23})x_2^2x_1\\
&+((1-m_{11})\gamma_{13}+(m_{12}-m_{11}-1)\gamma_{23})x _2^3.
\end{align*}

\begin{lemma}
The condition (\ref{condtion for comm}) is equivalent to the following equations hold
\begin{equation}\label{equation for Jordan}
\begin{aligned}
&(m_{11}-1)\gamma_{21}=0,\\
&(m_{11}-1)\gamma_{11}+(m_{11}-m_{12}+1)\gamma_{21}+(1-m_{11})\gamma_{22}=0,\\
&-(m_{11}+1)\gamma_{11}+(m_{11}-1)\gamma_{12}+(1-m_{11}+m_{12})\gamma_{21}+(m_{11}-m_{12})\gamma_{22}+(1-m_{11})\gamma_{23}=0,\\
&-2\gamma_{11}-\gamma_{12}+(m_{11}-1)\gamma_{13}+2\gamma_{21}+\gamma_{22}+(1-m_{11}-m_{12})\gamma_{23}=0.
\end{aligned}
\end{equation}
In this case, $\delta(r)=r\otimes \delta_r+\delta_l\otimes r$, where
\begin{eqnarray*}
&&\delta_{r}=(m_{11}\gamma_{11}+(m_{11}-m_{12})\gamma_{21}+\gamma_{22})x_1+(\gamma_{11}-\gamma_{21}+m_{11}\gamma_{23})x_2,\\
&&\delta_{l}=(\gamma_{11}-\gamma_{21})x_1+(\gamma_{11}+\gamma_{12}-\gamma_{21}-\gamma_{22}+m_{11}\gamma_{23})x_2.
\end{eqnarray*}
\end{lemma}
By (\ref{nakayama of comm}), we have the Nakayama automorphism of a graded Ore extension $B=A[z;\overline{\sigma},\overline{\delta}]$ satisfies that
 \begin{equation}\label{naka equa for jordan}
 \begin{aligned}
    &\mu_B(x_1)=m_{11}^{-1}x_1+(2m_{11}^{-1}-m_{11}^{-2}m_{12})x_2,\quad \mu_B(x_2)=m_{11}^{-1}x_2,\\
    &\mu_B(z)=m_{11}^2z+
    ((m_{11}+m_{11}^{-1})\gamma_{11}+(m_{11}-m_{11}^{-1}-m_{12})\gamma_{21}+\gamma_{22})x_1\\
    &\quad+((1+3m_{11}^{-1}-m_{11}^{-2}m_{12})\gamma_{11}+m_{11}^{-1}\gamma_{12}+(m_{11}^{-2}m_{12}-3m_{11}^{-1})\gamma_{21}-m_{11}^{-1}\gamma_{22}+(1+m_{11})\gamma_{23})x_2.
    \end{aligned}
\end{equation}
%The solutions to equations (\ref{equation for Jordan}) and the Nakayama automorphism of a graded Ore extension $B=A[z;\overline{\sigma},\overline{\delta}]$ are as follows.
\begin{solution}\label{solution: Jordan}
All solutions to equations (\ref{equation for Jordan}) are as follows.
\begin{enumerate}
\item If $m_{11}=1$, then
$
\gamma_{11}=(m_{12}\gamma_{21}+(1-m_{12})\gamma_{22})/2,\gamma_{12}=m_{12}\gamma_{22}-m_{12}\gamma_{23},
$ and $\gamma_{13}$ is free, where $\gamma_{21}=0$ if $m_{12}=2$ or $\gamma_{21}$ if free if $m_{12}\neq 2$;%, and
%\begin{align*}
%&\mu_B(x_1)=x_1+(2-m_{12})x_2,\quad \mu_B(x_2)=x_2,\\ &\mu_B(z)=z+(2-m_{12})\gamma_{22}x_1+((2-m_{12})(1-m_{12})\gamma_{22}/2+(2-m_{12})\gamma_{23})x_2.
%\end{align*}

\item If $m_{11}\neq 1$, $\gamma_{21}=0,\gamma_{11}=\gamma_{22},\gamma_{12}=(m_{11}-1)^{-1}(m_{12}+1)\gamma_{22}+\gamma_{23},
    \gamma_{13}=(m_{11}-1)^{-1}(m_{11}+m_{12})((m_{11}-1)^{-1}\gamma_{22}+\gamma_{23})$. % and
%    \begin{align*}
%\mu_B(x_1)=&m_{11}^{-1}x_1+(2m_{11}^{-1}-m_{11}^{-2}m_{12})x_2,\quad \mu_B(x_2)=m_{11}^{-1}x_2,\\ \mu_B(z)=&m_{11}^2z+(m_{11}+m_{11}^{-1}+1)\gamma_{22}x_1\\
%&+((m_{11}^2-m_{11})^{-1}(m_{11}^2+m_{11}+m_{11}^{-1}m_{12}-1)\gamma_{22}+(m_{11}+m_{11}^{-1}+1)\gamma_{23})x_2.
%\end{align*}
\end{enumerate}
\end{solution}

\begin{proof}[Proof of Theorem \ref{thm: CY for 2-dim}]
If $B$ is CY, then $\overline{\sigma}=\mu_A$ follows by Theorem \ref{thm: nakayama automorphism of ore extension}.

If $A$ is commutative, then $\mu_A$ is the identity map, that is $M=E_2$. Hence, (a) is an immediate result by Solution \ref{solution: commutative}(a) and (\ref{naka equa for comm}), or Theorem \ref{thm: NA of Ore ext over Polynomial}.

Now assume $A$ is noncommutative and  $\overline{\sigma}=\mu_A$.

If $A$ is a quantum plane, then $\mu_A(x_1)=qx_1$ and $\mu_A(x_2)=q^{-1}x_2$, that is $m_{11}=q,m_{22}=q^{-1},m_{12}=m_{21}=0$. Hence, $B$ is CY by Solution \ref{solution: q=-1}(a) and (\ref{naka equa for q=-1}) if $q=-1$, and Solution \ref{solution: qneq-1}(a) and (\ref{naka equa for qneq-1}) if $q\neq-1$.

If $A$ is the Jordan plane, then $\mu_A(x_1)=x_1+2x_2$ and $\mu_A(x_2)=x_2$, that is $m_{11}=m_{22}=1,m_{12}=2,m_{21}=0$. By Solution \ref{solution: Jordan}(a) and (\ref{naka equa for jordan}), one obtains that $B$ is CY.
%It is an immediate result by Solution \ref{solution: commutative}(a), Solution \ref{solution: q=-1}(a), Solution \ref{solution: q=-1 II}(a), Solution \ref{solution: qneq-1}(a), Solution \ref{solution: Jordan}(a).
\end{proof}

%It is a little bit surprised that there is no any restriction on the $\overline{\sigma}$-derivation $\overline{\delta}$ to make the graded Ore extension $A[z;\overline{\sigma},\overline{\delta}]$ be a CY algebra for any noncommutative noetherian 2-dimensional AS-regular algebra $A$. It is natural to ask whether this result holds for any noncommutative  AS-regular algebra of dimension $2$, or even for any noncommutative Koszul AS-regular algebra.

%\subsection{PBW deformations}
%
%d

\vskip7mm

\noindent {\bf Acknowledgments.} Y. Shen is supported by NSFC (Grant No.11701515) and the Fundamental Research Funds of Zhejiang Sci-Tech University (Grant No. 2019Q071).

\end{document}